\documentclass[10pt,reqno]{amsart}

\usepackage{amsthm, mathrsfs, amsmath, amstext, amsxtra, amsfonts, dsfont, amssymb, color}
\usepackage{lmodern}
\definecolor{green_dark}{rgb}{0,0.6,0}
\usepackage[colorlinks, linkcolor=red, citecolor=blue, urlcolor=green_dark, pagebackref, hypertexnames=false]{hyperref}

\newcommand{\N}{\mathbb N}

\newcommand{\R}{\mathbb R}
\newcommand{\C}{\mathbb C}
\newcommand{\nlsct}{\text{NLS}_c}
\newcommand{\nlsce}{\emph{NLS}_c}

\newcommand{\ep}{\epsilon}
\newcommand{\Gact} {\gamma_{\text{c}}}

\newcommand{\cbar}{\overline{c}}
\newcommand{\re}[1]{\mbox{Re} \ #1} 
\newcommand{\im}[1]{\mbox{Im} \ #1} 
\newcommand{\reemph}[1]{\mbox{\emph{Re}} \ #1} 
\newcommand{\imemph}[1]{\mbox{\emph{Im}} \ #1}
\newcommand{\scal}[1]{\left\langle #1 \right\rangle} 

\newcommand{\defendproof}{\hfill $\Box$} 

\newtheorem{theorem}{Theorem}[section]
\newtheorem{lem}[theorem]{Lemma} 
\newtheorem{prop}[theorem]{Proposition}
\newtheorem{coro}[theorem]{Corollary} 
\theoremstyle{definition}
\newtheorem{defi}[theorem]{Definition}
\newtheorem{rem}[theorem]{Remark}

\title[Global existence \& blowup NLS inverse-square potential]{Global existence and blowup for a class of the focusing nonlinear Schr\"odinger equation with inverse-square potential} 

\author[V. D. Dinh]{Van Duong Dinh}
\address[V. D. Dinh]{Institut de Math\'ematiques de Toulouse UMR5219, Universit\'e Toulouse CNRS, 31062 Toulouse Cedex 9, France and Department of Mathematics, HCMC University of Pedagogy, 280 An Duong Vuong, Ho Chi Minh, Vietnam}
\email{dinhvan.duong@math.univ-toulouse.fr}


\begin{document}

\maketitle
\begin{abstract}
We consider a class of the focusing nonlinear Schr\"odinger equation with inverse-square potential
\[
i\partial_t u + \Delta u -c|x|^{-2}u = - |u|^\alpha u, \quad u(0)=u_0 \in H^1, \quad (t,x)\in \R \times \R^d,
\]
where $d\geq 3$, $\frac{4}{d}\leq \alpha \leq \frac{4}{d-2}$ and $c\ne 0$ satisfies $c>-\lambda(d):=-\left(\frac{d-2}{2}\right)^2$. In the mass-critical case $\alpha=\frac{4}{d}$, we prove the global existence and blowup below ground states for the equation with $d\geq 3$ and $c>-\lambda(d)$. In the mass and energy intercritical case $\frac{4}{d}<\alpha<\frac{4}{d-2}$, we prove the global existence and blowup below the ground state threshold for the equation. This extends similar results of \cite{KillipMurphyVisanZheng} and \cite{LuMiaoMurphy} to any dimensions $d\geq 3$ and a full range $c>-\lambda(d)$. We finally prove the blowup below ground states for the equation in the energy-critical case $\alpha=\frac{4}{d-2}$ with $d\geq 3$ and $c>-\frac{d^2+4d}{(d+2)^2} \lambda(d)$.
\end{abstract}


\section{Introduction}
\setcounter{equation}{0}
Consider the Cauchy problem for the focusing nonlinear Schr\"odinger equation with inverse-square potential
\[
\left\{
\begin{array}{ccl}
i\partial_t u - P_c u &=& - |u|^\alpha u, \quad (t,x)\in \R \times \R^d, \\
u(0)&=& u_0,
\end{array} 
\right. \tag{$\nlsct$}
\]
where $u: \R \times \R^d \rightarrow \C, u_0:\R^d \rightarrow \C, d\geq 3, \alpha>0$ and $P_c=-\Delta + c|x|^{-2}$ with $c\ne 0$ satisfies $c>-\lambda(d):=-\left(\frac{d-2}{2}\right)^2$. The case $c=0$ is the well-known nonlinear Schr\"odinger equation which has been studied extensively over the last three decades. The nonlinear Schr\"odinger equation with inverse-square potential $(\nlsct)$ appears in a variety of physical settings and is of interest in quantum mechanics (see e.g. \cite{KalfSchminckeWalterWust} and references therein). The study of the $(\nlsct)$ has attracted a lot of interest in the past several years (see e.g. \cite{BurqPlanchonStalkerTahvildar-Zadeh, KillipMurphyVisanZheng, KillipMiaoVisanZhangZheng-sobolev, KillipMiaoVisanZhangZheng-energy, LuMiaoMurphy, OkazawaSuzukiYokota-cauchy, OkazawaSuzukiYokota-energy, Suzuki-hartree, Suzuki-bounded, ZhangZheng}).\newline
\indent The operator $P_c$ is the self-adjoint extension of $-\Delta + c|x|^{-2}$. It is well-known that in the range $-\lambda(d)<c<1-\lambda(d)$, the extension is not unique (see e.g. \cite{KalfSchminckeWalterWust}). In this case, we do make a choice among possible extensions, such as Friedrichs extension. The restriction on $c$ comes from the sharp Hardy inequality, namely
\begin{align}
\lambda(d) \int |x|^{-2} |u(x)|^2 dx \leq \int|\nabla u(x)|^2 dx, \quad \forall u \in H^1, \label{sharp hardy inequality}
\end{align}
which ensures that $P_c$ is a positive operator. \newline
\indent Throughout this paper, we denote for $\gamma \in \R$ and $q\in [1,\infty]$ the usual homogeneous and inhomogeneous Sobolev spaces associated to the Laplacian $-\Delta$ by $\dot{W}^{\gamma,q}$ and $W^{\gamma,q}$ respectively. We also use $\dot{H}^\gamma := \dot{W}^{\gamma,2}$ and $H^\gamma:=W^{\gamma,2}$. Similarly, we define the homogeneous Sobolev space $\dot{W}^{\gamma,q}_c$ associated to $P_c$ by the closure of $C^\infty_0(\R^d \backslash\{0\})$ under the norm
\begin{align*}
\|u\|_{\dot{W}^{\gamma,q}_c}:= \|\sqrt{P_c}^\gamma u\|_{L^q}. 
\end{align*}
The inhomogeneous Sobolev space associated to $P_c$ is defined by the closure of $C^\infty_0(\R^d)$ under the norm
\[
\|u\|_{W^{\gamma, q}_c}:=\|\scal{P_c}^\gamma u\|_{L^q},
\]
where $\scal{\cdot}$ is the Japanese bracket. We abbreviate $\dot{H}^\gamma_c: = \dot{W}^{\gamma,2}_c$ and $H^\gamma_c:=W^{\gamma,2}_c$. Note that by definition, we have
\begin{align}
\|u\|^2_{\dot{H}^1_c}=  \int |\nabla u(x)|^2 + c |x|^{-2}|u(x)|^2 dx. \label{H^1 norm P_c}
\end{align}
By the sharp Hardy inequality, we see that for $c>-\lambda(d)$,
\[
\|u\|_{\dot{H}^1_c}\sim \|u\|_{\dot{H}^1}.
\]
\indent Before stating our results, let us recall some facts for the $(\nlsct)$.  We first note that the $(\nlsct)$ is invariant under the scaling,
\[
u_\lambda(t,x):= \lambda^{\frac{2}{\alpha}} u(\lambda^2 t, \lambda x), \quad \lambda>0.
\]
An easy computation shows
\[
\|u_\lambda(0)\|_{\dot{H}^\gamma} = \lambda^{\gamma+\frac{2}{\alpha}-\frac{d}{2}} \|u_0\|_{\dot{H}^\gamma}.
\]
Thus, the critical Sobolev exponent is given by 
\begin{align}
\Gact := \frac{d}{2}-\frac{2}{\alpha}. \label{critical exponent}
\end{align}
Moreover, the $(\nlsct)$ has the following conserved quantities:
\begin{align}
M(u(t))&:= \int |u(t,x)|^2 dx = M(u_0), \label{mass conservation} \\
E_c(u(t)) &:= \int \frac{1}{2}|\nabla u(t,x)|^2 + \frac{c}{2} |x|^{-2} |u(t,x)|^2- \frac{1}{\alpha+2} |u(t,x)|^{\alpha+2} dx = E_c(u_0). \label{energy conservation} 
\end{align}
It is convenient to introduce the following numbers:
\begin{align}
\renewcommand*{\arraystretch}{1.2}
\alpha_\star:=\frac{4}{d}, \quad \alpha^\star:=\left\{
\begin{array}{cl}
\frac{4}{d-2} & \text{if } d\geq 3, \\
\infty &\text{if } d=1, 2.
\end{array}
\right. \label{define alpha star}
\end{align}
\indent The main purpose of this paper is to study the global existence and blowup for the $(\nlsct)$ in the mass-critical (i.e. $\alpha=\alpha_\star$), intercritical (mass-supercritical and energy-subcritical, i.e. $\alpha_\star<\alpha<\alpha^\star$) and energy-critical (i.e. $\alpha=\alpha^\star$) cases. 
\subsection{Mass-critical case} Let us first recall known results for the focusing mass-critical nonlinear Schr\"odinger equation, i.e. $c=0$ and $\alpha=\alpha_\star$ in $(\nlsct)$. One has the following (see e.g. \cite[Chapter 6]{Cazenave} for more details):
\begin{theorem} \label{theorem global blowup NLS mass-critical}
Let $u_0 \in H^1$ and $u$ be the corresponding solution to the mass-critical \emph{(NLS$_0$)} (i.e. $c=0$ and $\alpha=\alpha_\star$ in $(\nlsce)$). 
\begin{itemize}
\item[1.] Global existence \cite{Weinstein}: If $d\geq 1$ and $\|u_0\|_{L^2}<\|Q_0\|_{L^2}$, where $Q_0$ is the unique positive radial solution to the elliptic equation
\[
\Delta Q_0 -Q_0 +Q_0^{\alpha_\star+1}=0,
\]
then the solution $u$ exists globally in time and $\sup_{t\in \R}\|u(t)\|_{\dot{H}^1} <\infty$. 
\item[2.] Blowup \cite{OgawaTsutsumi1, OgawaTsutsumi2}: The solution $u$ blows up in finite time if one of the following conditions holds true:
\begin{enumerate}
\item[$\bullet$] $d\geq 1, E_0(u_0)<0$ and $xu_0 \in L^2$, 
\item[$\bullet$] $d\geq 2, E_0(u_0)<0$ and $u_0$ is radial, 
\item[$\bullet$] $d=1$ and $E_0(u_0)<0$.
\end{enumerate} 
\end{itemize}
\end{theorem} 
\begin{rem} \label{rem global blowup NLS mass-critical}
\begin{itemize}
\item[1.] By the sharp Gagliardo-Nirenberg inequality, the condition $\|u_0\|_{L^2} <\|Q_0\|_{L^2}$ implies $E_0(u_0)>0$. 
\item[2.] The condition $\|u_0\|_{L^2} <\|Q_0\|_{L^2}$ is sharp for the global existence in the sense that for any $M_0>\|Q_0\|_{L^2}$ (even for $M_0=\|Q_0\|_{L^2}$, see Item 4 below), there exists $u_0 \in H^1$ satisfying $\|u_0\|_{L^2} = M_0$ and the corresponding solution $u$ blows up in finite time.
\item[3.] The assumption $E_0(u_0)<0$ is a sufficient condition for finite time blowup but it is not necessary. One can show that for any $E_0>0$, there exists $u_0 \in H^1$ satisfying $E_0(u_0)=E_0$ and the corresponding solution blows up in finite time.
\item[4.] It is well-known (see e.g. \cite{Weinstein86} or \cite[Remark 6.7.3]{Cazenave}) that there exists a blowup solution to the mass-critical (NLS$_0$) with $\|u_0\|_{L^2} = \|Q_0\|_{L^2}$ by using the speudo-conformal transformation. 
\item[5.] Note also that Merle in \cite{Merle} proved the following classification of miminal mass blowup solutions for the mass-critical (NLS$_0$): Let $u_0 \in H^1$ be such that $\|u_0\|_{L^2}=\|Q_0\|_{L^2}$. If the corresponding solution blows up in finite time $0<T<+\infty$, then up to symmetries of the equation, $u(t,x)= S(t-T,x)$, where
\begin{align}
S(t,x):= \frac{1}{|t|^{\frac{d}{2}}} e^{-i\frac{|x|^2}{4t} +\frac{i}{t}} Q\left(\frac{x}{t}\right). \label{minimal mass blowup}
\end{align}
\end{itemize}
\end{rem}
Now let us consider $c\ne 0$ satisfy $c>-\lambda(d)$. Let $C_{\text{GN}}(c)$ be the sharp constant to the Gagliardo-Nirenberg inequality associated to the mass-critical $(\nlsct)$, namely,
\[
C_{\text{GN}}(c):=\sup\Big\{\|f\|^{\alpha_\star+2}_{L^{\alpha_\star+2}} \div \Big[\|f\|^{\alpha_\star}_{L^2} \|f\|^2_{\dot{H}^1_c} \Big] \ \Big| \ f \in H^1_c \backslash\{0\}  \Big\}.
\]
We will see in Theorem $\ref{theorem gagliardo nirenberg inequality}$ that:
\begin{itemize}
\item[1.] When $-\lambda(d)<c<0$, the sharp constant $C_{\text{GN}}(c)$ is attained by a non-negative radial solution to the elliptic equation
\begin{align*}
-P_c Q_c -Q_c + Q_c^{\alpha_\star+1}=0.
\end{align*}
\item[2.] When $c>0$, $C_{\text{GN}}(c)=C_{\text{GN}}(0)$, where $C_{\text{GN}}(0)$ is the sharp constant to the standard Gagliardo-Nirenberg inequality
\[
\|f\|^{\alpha_\star+2}_{L^{\alpha_\star+2}} \leq C_{\text{GN}}(0) \|f\|^{\alpha_\star}_{L^2} \|f\|^2_{\dot{H}^1}.
\]
However, $C_{\text{GN}}(c)$ is never attained. Moreover, if we restrict attention to the Gagliardo-Nirenberg inequality for radial functions, then the sharp constant for the radial Gagliardo-Nirenberg inequality associated to the mass-critical $(\nlsct)$, namely,
\[
C_{\text{GN}}(c, \text{rad}):= \sup\left\{\|f\|^{\alpha_\star +2}_{L^{\alpha_\star+2}} \div \Big[ \|f\|^{\alpha_\star}_{L^2} \|f\|^2_{\dot{H}^1_c}\Big] \ \Big| \ f \in H^1_c \backslash \{0\}, f \text{ radial}\right\}
\]
is attended by a radial solution $Q_{c,\text{rad}}$ to the elliptic equation
\[
-P_c Q_{c,\text{rad}} -Q_{c,\text{rad}} + Q_{c,\text{rad}}^{\alpha_\star+1}=0.
\]
Since $C_{\text{GN}}(c)$ is never attained, the constant $C_{\text{GN}}(c,\text{rad})$ is strictly smaller than $C_{\text{GN}}(c)$.
\end{itemize}
We will also see in Remark $\ref{rem ground state equation}$ that for $c>-\lambda(d)$,
\begin{align}
C_{\text{GN}}(c)= \frac{\alpha_\star+2}{2\|Q_{\cbar}\|_{L^2}^{\alpha_\star}}, \label{sharp constant mass-critical}
\end{align}
where $\cbar:= \min\{c,0\}$. Moreover, for $c>0$,
\begin{align}
C_{\text{GN}}(c,\text{rad})= \frac{\alpha_\star+2}{2\|Q_{c,\text{rad}}\|_{L^2}^{\alpha_\star}}, \label{sharp constant mass-critical radial}
\end{align}
Our first result is the following global existence and blowup for the mass-critical $(\nlsct)$.
\begin{theorem} \label{theorem global blowup NLS mass-critical inverse square}
Let $d\geq 3$ and $c\ne 0$ be such that $c>-\lambda(d)$. Let $u_0 \in H^1$ and $u$ be the corresponding solution to the mass-critical $(\nlsce)$ (i.e. $\alpha=\alpha_\star$).
\begin{itemize}
\item[1.] If $\|u_0\|_{L^2}<\|Q_{\cbar}\|_{L^2}$, then the solution $u$ exists globally and $\sup_{t\in \R} \|u(t)\|_{\dot{H}^1_c}<\infty$. 
\item[2.] If $E_c(u_0)<0$ and either $x u_0 \in L^2$ or $u_0$ is radial, then the solution $u$ blows up in finite time.
\end{itemize}
\end{theorem}
\begin{rem}\label{rem global blowup NLS mass-critical inverse square}
\begin{itemize}
\item[1.] In \cite{CsoboGenoud}, Csobo-Genoud proved the global existence for the mass-critical $(\nlsct)$ with $-\lambda(d)<c<0$ under the assumption $\|u_0\|_{L^2}<\|Q_c\|_{L^2}$. Here we extend their result to any $c \ne 0$ and $c>-\lambda(d)$.
\item[2.] By the sharp Gagliardo-Nirenberg inequality associated to $(\nlsct)$, we see that the condition $\|u_0\|_{L^2}<\|Q_{\cbar}\|_{L^2}$ implies that $E_c(u_0)>0$. Indeed, applying the sharp Gagliardo-Nirenberg inequality and $(\ref{sharp constant mass-critical})$,
\begin{align*}
E_c(u_0) &= \frac{1}{2}\|u_0\|^2_{\dot{H}^1_c} -\frac{1}{\alpha_\star+2} \|u_0\|^{\alpha_\star+2}_{L^{\alpha_\star+2}} \\
&\geq \frac{1}{2}\|u_0\|^2_{\dot{H}^1_c} -\frac{1}{\alpha_\star+2} C_{\text{GN}}(c) \|u_0\|_{L^2}^{\alpha_\star} \|u_0\|^2_{\dot{H}^1_c} \\
&\geq \frac{1}{2}\|u_0\|^2_{\dot{H}^1_c} \Big[1-\Big(\frac{\|u_0\|_{L^2}}{\|Q_{\cbar}\|_{L^2}}\Big)^{\alpha_\star} \Big]>0.
\end{align*}
\item[3.] When $-\lambda(d)<c<0$, the condition $\|u_0\|_{L^2}<\|Q_{\cbar}\|_{L^2}=\|Q_c\|_{L^2}$ is sharp for the global existence. In fact, for any $M_c>\|Q_c\|_{L^2}$ (even for $M_c=\|Q_c\|_{L^2}$, see Item 5 below), we can show (see Remark $\ref{rem proof global NLS mass-critical inverse square}$) that there exists $u_0 \in H^1$ satisfying $\|u_0\|_{L^2} = M_c$ and the corresponding solution $u$ to the mass-critical $(\nlsct)$ blows up in finite time. When $c>0$, the condition $\|u_0\|_{L^2}<\|Q_0\|_{L^2}$ is not sharp. Indeed, if $u_0$ is radial and satisfies $\|u_0\|_{L^2}<\|Q_{c,\text{rad}}\|_{L^2}$, then the corresponding solution exists globally. Note that $\|Q_{c,\text{rad}}\|_{L^2}>\|Q_0\|_{L^2}$. Moreover, for any $M_c>\|Q_{c,\text{rad}}\|_{L^2}$ (even for $M_c=\|Q_{c,\text{rad}}\|_{L^2}$, see again Item 5 below), we can show (see again Remark $\ref{rem proof global NLS mass-critical inverse square}$) that there exists $u_0\in H^1$ radial satisfying $\|u_0\|_{L^2}=M_c$ and the corresponding solution blows up in finite time.
\item[4.] The condition $E_c(u_0)<0$ is a sufficient condition for finite time blowup, but it is not necessary. We will see in Remark $\ref{rem necessary condition blowup mass-critical inverse square}$ that for any $E_c>0$, there exists $u_0 \in H^1$ satisfying $E_c(u_0)=E_c$ and the corresponding solution blows up in finite time.
\item[5.] Recently, Csobo-Genoud in \cite[Lemma 1]{CsoboGenoud} made use of the speudo-conformal transformation to show that for $-\lambda(d)<c<0$, there exists a blowup solution to the mass-critical $(\nlsct)$ with $\|u_0\|_{L^2} = \|Q_c\|_{L^2}$. By a similar argument, we can show (see Remark $\ref{rem existence blowup solution mass-critical radial}$) that for $c>0$, there exists a radial blowup solution to the mass-critical $(\nlsct)$ with $\|u_0\|_{L^2}=\|Q_{c,\text{rad}}\|_{L^2}$.
\item[6.] In \cite{CsoboGenoud}, they also proved the classification of miminal mass blowup solutions for the mass-critical $(\nlsct)$ with $-\lambda(d)<c<0$. Their result is as follows: Let $u_0 \in H^1$ be such that $\|u_0\|_{L^2}=\|Q_c\|_{L^2}$. If the corresponding solution blows up in finite time $0<T<+\infty$, then up to symmetries of the equation \footnote{The $(\nlsct)$ does not enjoy the space translation invariance and the Galilean invariance.}, $u(t,x)= S(t-T,x)$, where $S$ is as in $(\ref{minimal mass blowup})$. We expect that a similar result should hold for radial blowup solutions with $c>0$. 
\end{itemize}
\end{rem}
\subsection{Intercritical case}
We next consider the intercritical (i.e. mass-supercritical and energy-subcritical) case. Let us recall known results for the focusing intercritical nonlinear Schr\"odinger equation, i.e. $c=0$ and $\alpha_\star<\alpha<\alpha^\star$ in $(\nlsct)$. The global existence, scattering and blowup were studied in \cite{HolmerRoudenko, DuyckaertsHolmerRoudenko, FangXieCazenave}. In order to state these results, let us define the following quantities:
\[
H(0):=E_0(Q_0) M(Q_0)^\sigma, \quad K(0):= \|Q_0\|_{\dot{H}^1} \|Q_0\|^\sigma_{L^2},
\]
where 
\begin{align}
\sigma:=\frac{1-\Gact}{\Gact}=\frac{4-(d-2)\alpha}{d\alpha-4}. \label{define sigma}
\end{align}
and $Q_0$ is the unique positive radial solution to the elliptic equation
\[
\Delta Q_0 -Q_0 + Q_0^{\alpha+1}=0.
\]
\begin{theorem}[\cite{HolmerRoudenko, DuyckaertsHolmerRoudenko, FangXieCazenave}] \label{theorem global existence and scattering c=0}
Let $d\geq 1, u_0 \in H^1$ and $u$ be the corresponding solution to the intercritical \emph{(NLS$_0$)} (i.e. $c=0$ and $\alpha_\star<\alpha<\alpha^\star$ in $(\nlsce)$). Suppose that $E_0(u_0) M(u_0)^\sigma< H(0)$.
\begin{itemize}
\item[1.] If $\|u_0\|_{\dot{H}^1} \|u_0\|^\sigma_{L^2} < K(0)$, then the solution $u$ exists globally in time and 
\[
\|u(t)\|_{\dot{H}^1} \|u(t)\|^\sigma_{L^2} < K(0),
\]
for any $t\in \R$. Moreover, the solution $u$ scatters in $H^1$.
\item[2.] If $\|u_0\|_{\dot{H}^1} \|u_0\|^\sigma_{L^2} > K(0)$ and either 
\begin{itemize}
\item[$\bullet$]  $xu_0 \in L^2$,
\item[$\bullet$] or $d\geq 3$, $u_0$ is radial,
\item[$\bullet$] or $d=2$, $u_0$ is radial and $\alpha_\star < \alpha \leq 4$, 
\end{itemize}
then the solution $u$ blows up in finite time and
\[
\|u(t)\|_{\dot{H}^1} \|u(t)\|^\sigma_{L^2} > K(0),
\]
for any $t$ in the existence time. Moreover, the finite time blowup still holds true if in place of $E_0(u_0) M(u_0)^\sigma<H(0)$ and $\|u_0\|_{\dot{H}^1} \|u_0\|^\sigma_{L^2}>K(0)$, we assume that $E(u_0)<0$. 
\end{itemize}
\end{theorem}
Now let $c \ne 0$ be such that $c>-\lambda(d)$, and let $C_{\text{GN}}(c)$ be the sharp constant in the Gagliardo-Nirenberg inequality associated to the intercritical $(\nlsct)$, namely,
\[
C_{\text{GN}}(c):=\sup\Big\{ \|f\|^{\alpha+2}_{L^{\alpha+2}} \div \Big[\|f\|^{\frac{4-(d-2)\alpha}{2}} \|f\|^{\frac{d\alpha}{2}}_{\dot{H}^1_c} \Big] \ \Big| \ f \in H^1_c \backslash \{0\} \Big\}.
\]
We will see in Theorem $\ref{theorem gagliardo nirenberg inequality}$ that:
\begin{itemize}
\item[1.] When $-\lambda(d)<c<0$, the sharp constant $C_{\text{GN}}(c)$ is attained by a solution $Q_c$ to the elliptic equation
\[
-P_c Q_c -Q_c + Q_c^{\alpha+1}=0.
\]
\item[2.] When $c>0$, $C_{\text{GN}}(c)=C_{\text{GN}}(0)$, where $C_{\text{GN}}(0)$ is again the sharp constant to the standard Gagliardo-Nirenberg inequality
\[
\|f\|^{\alpha+2}_{L^{\alpha+2}} \leq C_{\text{GN}}(0) \|f\|^{\frac{4-(d-2)\alpha}{2}}_{L^2} \|f\|^{\frac{d\alpha}{2}}_{\dot{H}^1}.
\]
Moreover, $C_{\text{GN}}(c)$ is never attained. However, if we restrict attention to the Gagliardo-Nirenberg inequality for radial functions, then the sharp constant for the radial Gagliardo-Nirenberg inequality associated to the intercritical $(\nlsct)$, namely,
\[
C_{\text{GN}}(c, \text{rad}):= \sup\left\{ \|f\|^{\alpha+2}_{L^{\alpha+2}} \div \Big[\|f\|^{\frac{4-(d-2)\alpha}{2}} \|f\|^{\frac{d\alpha}{2}}_{\dot{H}^1_c} \Big] \ \Big| \ f \in H^1_c \backslash \{0\},  f \text{ radial}\right\}
\]
is attended by a radial solution $Q_{c,\text{rad}}$ to the elliptic equation
\[
-P_c Q_{c,\text{rad}} -Q_{c,\text{rad}} + Q_{c,\text{rad}}^{\alpha+1}=0.
\]
Since $C_{\text{GN}}(c)$ is never attained, the constant $C_{\text{GN}}(c,\text{rad})$ is strictly smaller than $C_{\text{GN}}(c)$.
\end{itemize} 
We define the following quantities:
\begin{align}
H(c):= E_{\cbar}(Q_{\cbar}) M(Q_{\cbar})^\sigma, \quad K(c):= \|Q_{\cbar}\|_{\dot{H}^1_{\cbar}} \|Q_{\cbar}\|^\sigma_{L^2}, \label{define Hc Kc}
\end{align}
where $\cbar = \min\{c, 0\}$. Our next result is the following global existence and blowup for the intercritical $(\nlsct)$.
\begin{theorem} \label{theorem global blowup NLS intercritical inverse square}
Let $d\geq 3, \alpha_\star<\alpha<\alpha^\star$ and $c\ne 0$ be such that $c>-\lambda(d)$. Let $u_0 \in H^1$ and $u$ be the corresponding solution of the intercritical $(\nlsce)$ (i.e. $\alpha_\star<\alpha<\alpha^\star$). Suppose that
\begin{align}
E_c(u_0) M(u_0)^\sigma < H(c). \label{condition below ground state NLS intercritical inverse square}
\end{align} 
\begin{itemize}
\item[1.] Global existence: If 
\begin{align}
\|u_0\|_{\dot{H}^1_c} \|u_0\|^\sigma_{L^2} < K(c), \label{condition global existence NLS intercritical inverse square}
\end{align}
then the solution $u$ exists globally in time and 
\begin{align}
\|u(t)\|_{\dot{H}^1_c} \|u(t)\|_{L^2}^\sigma < K(c). \label{property global solution NLS intercritical inverse square}
\end{align}
for any $t\in \R$.
\item[2.] Blowup: If 
\begin{align}
\|u_0\|_{\dot{H}^1_c} \|u_0\|^\sigma_{L^2} > K(c), \label{condition blowup NLS intercritical inverse square}
\end{align}
and either $xu_0 \in L^2$ or $u_0$ is radial, then the solution $u$ blows up in finite time and
\begin{align}
\|u(t)\|_{\dot{H}^1_c} \|u(t)\|^\sigma_{L^2} > K(c), \label{property blowup solution NLS intercritical inverse square}
\end{align}
for any $t$ in the existence time. Moreover, the finite time blowup still holds true if in place of $(\ref{condition below ground state NLS intercritical inverse square})$ and $(\ref{condition blowup NLS intercritical inverse square})$, we assume that $E_c(u_0)<0$. 
\end{itemize}
\end{theorem}
\begin{rem}\label{rem global blowup NLS intercritical inverse square}
\begin{itemize}
\item[1.] 
In \cite{KillipMurphyVisanZheng}, Killip-Murphy-Visan-Zheng considered the cubic $(\nlsct)$ in $3D$ (i.e. $\alpha=2$ and $c>-\frac{1}{4}$) and proved that the global existence as well as scattering hold true under the assumptions $(\ref{condition below ground state NLS intercritical inverse square}), (\ref{condition global existence NLS intercritical inverse square})$ and the blowup holds true under the assumptions $(\ref{condition below ground state NLS intercritical inverse square}), (\ref{condition blowup NLS intercritical inverse square})$. Recently, Lu-Miao-Murphy in \cite{LuMiaoMurphy} proved a similar result as in \cite{KillipMurphyVisanZheng} for the intercritical $(\nlsct)$ with
\[
\renewcommand*{\arraystretch}{1.2}
\left\{
\begin{array}{cl}
c>-\frac{1}{4} &\text{if } d=3, \quad \frac{4}{3}<\alpha\leq 2 \\
c>-\lambda(d) + \left(\frac{d-2}{2}-\frac{1}{\alpha}\right)^2 &\text{if } 3\leq d\leq 6, \quad \max\left\{\frac{2}{d-2}, \frac{4}{d} \right\} <\alpha<\frac{4}{d-2}. 
\end{array}
\right.
\]
Here we extend the global existence and blowup results of \cite{KillipMurphyVisanZheng, LuMiaoMurphy} to any dimensions $d\geq3$ and the full range $c>-\lambda(d)$. We expect that the global solution in Theorem $\ref{theorem global blowup NLS intercritical inverse square}$ scatters in $H^1$ under a certain restriction on $c$. Note that the scattering of global solutions depends heavily on Strichartz estimates which were proved in \cite{BurqPlanchonStalkerTahvildar-Zadeh, BoucletMizutani}. In order to successfully apply Strichartz estimates, we need the equivalence of Sobolev norms between the ones associated to $P_c$ and those associated to $-\Delta$ (see Subsection $\ref{subsection equivalent sobolev norms}$ for more details). This will lead to a restriction on the validity of $c$.    
\item[2.] Theorem $\ref{theorem global blowup NLS intercritical inverse square}$ says that the condition $(\ref{condition global existence NLS intercritical inverse square})$ is sharp for the global existence except for the threshold level
\[
\|u_0\|_{\dot{H}^1_c}\|u_0\|^\sigma_{L^2} = K(c).
\]
It is an interesting open problem to show that there exists blowup solutions to the intercritical (NLS$_0$) and $(\nlsct)$ equations at this threshold.
\item[3.] It is worth mentioning that if the energy of the initial data is negative, then $(\ref{condition below ground state NLS intercritical inverse square})$ is always satisfied. Indeed, we will see in $(\ref{energy ground state})$ that
\[
E(Q_{\cbar}) =\frac{d\alpha-4}{2(4-(d-2)\alpha)}\|Q_{\cbar}\|^2_{L^2}=\frac{d\alpha-4}{2d\alpha} \|Q_{\cbar}\|^2_{\dot{H}^1_{\cbar}},
\] 
hence $H(c)$ is always non-negative.
\end{itemize}
\end{rem}
In the case $c>0$, we have the following improved result for radial solutions. 
\begin{theorem} \label{theorem global blowup NLS intercritical inverse square radial}
Let $d\geq 3, \alpha_\star<\alpha<\alpha^\star$ and $c>0$. Let $u_0 \in H^1$ be radial and $u$ the corresponding solution of the intercritical $(\nlsce)$ (i.e. $\alpha_\star<\alpha<\alpha^\star$). Suppose that
\begin{align}
E_c(u_0) M(u_0)^\sigma < H(c, \emph{rad})=:E_c(Q_{c,\emph{rad}}) M(Q_{c,\emph{rad}})^\sigma. \label{condition below ground state NLS intercritical inverse square radial}
\end{align} 
\begin{itemize}
\item[1.] Global existence: If 
\begin{align}
\|u_0\|_{\dot{H}^1_c} \|u_0\|^\sigma_{L^2} < K(c,\emph{rad})=:\|Q_{c,\emph{rad}}\|_{\dot{H}^1_c} \|Q_{c,\emph{rad}}\|^\sigma_{L^2}, \label{condition global existence NLS intercritical inverse square radial}
\end{align}
then the solution $u$ exists globally in time and 
\begin{align}
\|u(t)\|_{\dot{H}^1_c} \|u(t)\|_{L^2}^\sigma < K(c,\emph{rad}). \label{property global solution NLS intercritical inverse square radial}
\end{align}
for any $t\in \R$.
\item[2.] Blowup: If 
\begin{align}
\|u_0\|_{\dot{H}^1_c} \|u_0\|^\sigma_{L^2} > K(c,\emph{rad}), \label{condition blowup NLS intercritical inverse square radial}
\end{align}
then the solution $u$ blows up in finite time and
\begin{align}
\|u(t)\|_{\dot{H}^1_c} \|u(t)\|^\sigma_{L^2} > K(c,\emph{rad}), \label{property blowup solution NLS intercritical inverse square radial}
\end{align}
for any $t$ in the existence time. Moreover, the finite time blowup still holds true if in place of $(\ref{condition below ground state NLS intercritical inverse square radial})$ and $(\ref{condition blowup NLS intercritical inverse square radial})$, we assume that $E_c(u_0)<0$. 
\end{itemize}
\end{theorem}
Since $C_{\text{GN}}(c, \text{rad})< C_{\text{GN}}(c)$, we will see in Remark $\ref{rem ground state equation}$ that $H(c)<H(c,\text{rad})$ and $K(c)<K(c,\text{rad})$. This shows that the class of radial solutions enjoys strictly larger thresholds for the global existence and the blowup. 
\subsection{Energy-critical case} We finally consider the energy-critical case. As above, we recall known results for the focusing energy-critical nonlinear Schr\"odinger equation, i.e. $c=0$ and $\alpha=\alpha^\star$ in $(\nlsct)$. The global existence, scattering and blowup for the energy-critical (NLS$_0$) were first studied in \cite{KenigMerle} where Kenig-Merle proved the global existence, scattering and blowup for the equation under the radial assumption of initial data in dimensions $d=3, 4, 5$. This was extended to dimensions $d\geq 3$ in \cite{KillipVisanZhang-unpublish}. Later, Killip-Visan in \cite{KillipVisan} proved the global existence and scattering for the equation with general (non-radial) data in dimensions five and higher. They also proved the existence of blowup solutions in dimensions $d\geq 3$. The global existence and scattering for the energy-critical (NLS$_0$) for general data still remain open for $d=3, 4$. To state their results, we recall the following facts. Let 
\begin{align}
W_0(x):=\Big(1+ \frac{|x|^2}{d(d-2)}\Big)^{-\frac{d-2}{2}}. \label{define W_0}
\end{align}
It is well-known that $W$ solves the elliptic equation
\[
\Delta W_0 + |W_0|^{\alpha^\star} W_0=0.
\]
In particular, $W_0$ is a stationary solution to the energy-critical (NLS$_0$). Note that $W_0 \in \dot{H}^1$ but it need not belong to $L^2$. 
\begin{theorem}[\cite{KenigMerle}] \label{theorem global scattering blowup NLS energy-critical Kenig-Merle}
Let $d=3, 4, 5$. Let $u_0 \in \dot{H}^1$ be radial and $u$ be the corresponding solution to the energy-critical \emph{(NLS$_0$)} (i.e. $c=0$ and $\alpha=\alpha^\star$ in $(\nlsce)$). Suppose that $E_0(u_0)<E_0(W_0)$.
\begin{itemize}
\item[1.] If $\|u_0\|_{\dot{H}^1} <\|W_0\|_{\dot{H}^1}$, then the solution $u$ exists globally and scatters in $\dot{H}^1$.
\item[2.] If $\|u_0\|_{\dot{H}^1} >\|W_0\|_{\dot{H}^1}$ and either $xu_0 \in L^2$ or $u_0 \in H^1$ is radial, then the solution $u$ blows up in finite time. Moreover, the finite time blowup still holds true if in place of $E_0(u_0) <E_0(W_0)$ and $\|u_0\|_{\dot{H}^1} >\|W_0\|_{\dot{H}^1}$, we assume that $E_0(u_0)<0$.
\end{itemize}
\end{theorem}
\begin{theorem} \label{theorem global scattering blowup NLS energy-critical Killip-Visan}
Let $u_0 \in \dot{H}^1$ and $u$ be the corresponding solution to the energy-critical \emph{(NLS$_0$)}. Suppose that $E_0(u_0)<E_0(W_0)$.
\begin{itemize}
\item[1.] \cite{KillipVisan, Dodson} If $d\geq 4$ and $\|u_0\|_{\dot{H}^1} <\|W_0\|_{\dot{H}^1}$, then the solution $u$ exists globally and scatters in $\dot{H}^1$.
\item[2.] \cite{KillipVisan} If $d\geq 3, \|u_0\|_{\dot{H}^1} >\|W_0\|_{\dot{H}^1}$ and either $xu_0 \in L^2$ or $u_0 \in H^1$ is radial, then the solution $u$ blows up in finite time. Moreover, the finite time blowup still holds true if in place of $E_0(u_0) <E_0(W_0)$ and $\|u_0\|_{\dot{H}^1} >\|W_0\|_{\dot{H}^1}$, we assume that $E_0(u_0)<0$.
\end{itemize}
\end{theorem}
\begin{rem} \label{rem incompatible blowup condition}
Note that the conditions $E_0(u_0)<E_0(W_0)$ and $\|u_0\|_{\dot{H}^1} =\|W_0\|_{\dot{H}^1}$ are incompatible.
\end{rem}
Now let $c\ne 0$ satisfy $c>-\lambda(d)$, and let $C_{\text{SE}}(c)$ be the sharp constant in the Sobolev embedding inequality associated to the energy-critical $(\nlsct)$, namely,
\[
C_{\text{SE}}(c):=\sup \left\{ \|f\|_{L^{\alpha^\star+2}} \div \|f\|_{\dot{H}^1_c} \ | \ f \in \dot{H}^1_c \backslash \{0\} \right\}.
\]
We will see in Theorem $\ref{theorem sobolev embedding inequality}$ that:
\begin{itemize}
\item[1.] When $-\lambda(d)<c<0$, the sharp constant $C_{\text{SE}}(c)$ is attained by functions $f(x)$ of the form $\lambda W_c (\mu x)$ for some $\lambda \in \C$ and $\mu>0$, where
\begin{align}
W_c(x) := [d(d-2) \beta^2]^{\frac{d-2}{4}} \Big[\frac{|x|^{\beta-1}}{1+|x|^{2\beta}} \Big]^{\frac{d-2}{2}}, \label{define W_c}
\end{align}
with $\beta = 1-\frac{2\rho}{d-2}$ (see $(\ref{define rho})$ for the definition of $\rho$). 
\item[2.] When $c>0$, $C_{\text{SE}}(c)=C_{\text{SE}}(0)$, where $C_{\text{SE}}(0)$ is the sharp constant to the standard Sobolev embedding inequality
\[
\|f\|_{L^{\alpha^\star+2}} \leq C_{\text{SE}}(0) \|f\|_{\dot{H}^1}.
\]
Moreover, $C_{\text{SE}}(c)$ is never attained. Note that the constant $C_{\text{SE}}(0)$ is attained by functions $f(x)$ of a form $\lambda W_0(\mu x + y)$ for some $\lambda \in \C, y\in \R^d$ and $\mu>0$. However, if we restrict attention to radial functions, then the sharp constant for the radial Sobolev embedding associated to the energy-critical $(\nlsct)$, namely, 
\[
C_{\text{SE}}(c,\text{rad}):=\sup \left\{ \|f\|_{L^{\alpha^\star+2}} \div \|f\|_{\dot{H}^1_c} \ | \ f \in \dot{H}^1_c \backslash \{0\}, f \text{ radial} \right\}
\]
is attained by functions $f(x)$ of the form $\lambda W_c(\mu x)$ for some $\lambda \in \C$ and $\mu >0$.
\end{itemize}
Our last result concerns with the blowup for the energy-critical $(\nlsct)$.
\begin{theorem} \label{theorem blowup NLS energy-critical inverse square}
Let $d\geq 3$ and $c \ne 0$ be such that $c>-\frac{d^2+4d}{(d+2)^2} \lambda(d)$. Let $u_0 \in \dot{H}^1$ and $u$ be the corresponding solution to the energy-critical $(\nlsce)$ (i.e. $\alpha=\alpha^\star$). Suppose that either $E_c(u_0)<0$, or if $E_c(u_0) \geq 0$, we assume that $E_c(u_0) <E_{\cbar}(W_{\cbar})$ and $\|u_0\|_{\dot{H}^1_c} >\|W_{\cbar}\|_{\dot{H}^1_{\cbar}}$, where $\cbar = \min\{c,0\}$. If $xu_0 \in L^2$ or $u_0$ is radial, then the solution $u$ blows up in finite time.
\end{theorem}
\begin{rem} \label{rem blowup NLS energy-critical inverse square}
\begin{itemize}
\item[1.] As in Remark $\ref{rem incompatible blowup condition}$, the conditions $E_c(u_0) <E_{\cbar}(W_{\cbar})$ and $\|u_0\|_{\dot{H}^1_c} =\|W_{\cbar}\|_{\dot{H}^1_{\cbar}}$ are incompatible.
\item[2.] Theorem $\ref{theorem blowup NLS energy-critical inverse square}$ was stated in \cite{KillipMiaoVisanZhangZheng-energy} without proof. In this paper, we give a proof for this result. The restriction of $c$ comes from the local theory via Strichartz estimates (see Proposition $\ref{prop local wellposedness energy-critical inverse square}$). 
\item[3.] We expect that the global existence as well as scattering for the energy-critical $(\nlsct)$ hold true for $u_0 \in \dot{H}^1$ satisfying $E_c(u_0) <E_{\cbar}(W_{\cbar})$ and $\|u_0\|_{\dot{H}^1_c} <\|W_{\cbar}\|_{\dot{H}^1_{\cbar}}$. It is a delicate open problem. 
\end{itemize}
\end{rem}
In the case $c>0$, we have the following blowup result for radial solutions. 
\begin{theorem} \label{theorem blowup NLS energy-critical inverse square radial}
Let $d\geq 3$ and $c >0$. Let $u_0 \in \dot{H}^1$ radial and $u$ be the corresponding solution to the energy-critical $(\nlsce)$ (i.e. $\alpha=\alpha^\star$). Suppose that either $E_c(u_0)<0$, or if $E_c(u_0)\geq 0$, we assume that $E_c(u_0) <E_{c}(W_{c})$ and $\|u_0\|_{\dot{H}^1_c} >\|W_{c}\|_{\dot{H}^1_{c}}$. Then the solution $u$ blows up in finite time.
\end{theorem}
Since $C_{\text{GN}}(c)>C_{\text{GN}}(c,\text{rad})$, we have from $(\ref{relation sharp sobolev embedding constant})$ and $(\ref{relation sharp sobolev embedding constant radial})$ that $E_{0}(W_0)<E_c(W_c)$. This shows that the blowup threshold for radial solutions is strictly larger than the one for non-radial solutions.\newline
\indent The paper is organized as follows. In Section 2, we recall some preliminary results related to the $(\nlsct)$. In Section 3, we recall the local well-posedness for the $(\nlsct)$ in the energy-subcritical and energy-critical cases. In Section 4, we recall the sharp Gagliardo-Nirenberg inequality and the sharp Sobolev embedding inequality for the $(\nlsct)$ by using the variational analysis. We next derive the standard virial identity as well as the localized virial estimate in Section 5. Section 6 is devoted to the proofs of global existence results. Finally, we give the proofs of blowup results in Section 7.
\section{Preliminaries} \label{section preliminaries}
\setcounter{equation}{0}
In the sequel, the notation $A \lesssim B$ denotes an estimate of the form $A\leq CB$ for some constant $C>0$. The notation $A \sim B$ means $A\lesssim B$ and $B \lesssim A$. The various constant $C$ may change from line to line. 
\subsection{Strichartz estimates} \label{subsection strichartz estimates}
Let $J \subset \R$ and $p, q \in [1,\infty]$. We define the mixed norm
\[
\|u\|_{L^p(J, L^q)} := \Big( \int_J \Big( \int_{\R^d} |u(t,x)|^q dx \Big)^{\frac{1}{q}} \Big)^{\frac{1}{p}}
\] 
with a usual modification when either $p$ or $q$ are infinity.
\begin{defi}
A pair $(p,q)$ is said to be \textbf{Schr\"odinger admissible}, for short $(p,q) \in S$, if 
\[
(p,q) \in [2,\infty]^2, \quad (p,q,d) \ne (2,\infty,2), \quad \frac{2}{p}+\frac{d}{q} = \frac{d}{2}.
\]
\end{defi}
We recall Strichartz estimates for the inhomogeneous Schr\"odinger equation with inverse-square potential. 
\begin{prop}[Strichartz estimates \cite{BurqPlanchonStalkerTahvildar-Zadeh, BoucletMizutani}] \label{prop strichartz esimates}
Let $d\geq 3$ and $c>-\lambda(d)$. Let $u$ be a solution to the inhomogeneous Schr\"odinger equation with inverse-square potential, namely
\[
u(t)= e^{itP_c}u_0 + \int_0^t e^{i(t-s)P_c} F(s) ds,
\]
for some data $u_0, F$. Then, for any $(p,q), (a,b) \in S$, 
\begin{align}
\|u\|_{L^p(\R, L^q)} \lesssim \|u_0\|_{L^2} + \|F\|_{L^{a'}(\R, L^{b'})}. \label{strichartz estimates}
\end{align}
Moreover, for any $\gamma \in \R$, $(p,q), (a,b) \in S$, 
\begin{align}
\|u\|_{L^p(\R, \dot{W}^{\gamma,q}_c)} \lesssim \|u_0\|_{\dot{H}^\gamma_c} + \|F\|_{L^{a'}_t(\R, \dot{W}^{\gamma, b'}_c)}. \label{strichartz estimates P_c}
\end{align}
Here $(a,a')$ and $(b, b')$ are conjugate pairs.
\end{prop}
Note that Strichartz estimates for the homogeneous nonlinear Schr\"odinger equation with inverse-square potential were first proved by Burq-Planchon-Stalker-Zadeh in \cite{BurqPlanchonStalkerTahvildar-Zadeh} except the endpoint $(p,q)=(2,\frac{2d}{d-2})$. Recently, Bouclet-Mizutani in \cite{BoucletMizutani} proved Strichartz estimates with the full set of Schr\"odinger admissible pairs for the homogeneous and inhomogeneous nonlinear Schr\"odinger equation with critical potentials including the inverse-square potential. We refer the reader to \cite{BurqPlanchonStalkerTahvildar-Zadeh, BoucletMizutani} for more details. 
\subsection{Equivalence of Sobolev norms} \label{subsection equivalent sobolev norms}
In this subsection, we recall the equivalence between Sobolev norms defined by $P_c$ and the ones defined by the usual Laplacian $-\Delta$. In \cite[Proposition 1]{BurqPlanchonStalkerTahvildar-Zadeh}, Burq-Planchon-Stalker-Zadel proved the following:
\[
\|u\|_{\dot{H}^\gamma_c} \sim \|u\|_{\dot{H}^\gamma}, \quad \forall \gamma \in [-1,1].
\]
Later, Zhang-Zheng in \cite{ZhangZheng} extended this result to homogeneous Sobolev spaces $\dot{W}^{\gamma,q}_c$ and $\dot{W}^{\gamma, q}$ for $0\leq \gamma\leq 1$ and a certain range of $q$. 
Recently, Killip-Miao-Visan-Zhang-Zheng extended these results to a more general setting. To state their result, let us introduce
\begin{align}
\rho:= \frac{d-2}{2}-\sqrt{\left( \frac{d-2}{2}\right)^2+c }. \label{define rho}
\end{align}
\begin{prop}[Equivalence of Sobolev norms \cite{KillipMiaoVisanZhangZheng-sobolev}] \label{prop equivalence sobolev spaces}
Let $d\geq 3, c\geq -\lambda(d), 0<\gamma<2$ and $\rho$ be as in $(\ref{define rho})$.
\begin{itemize}
\item[1.] If $1<q<\infty$ satisfies $\frac{\gamma+\rho}{d} <\frac{1}{q} <\min\left\{1, \frac{d-\rho}{d} \right\}$, then
\[
\|f\|_{\dot{W}^{\gamma, q}} \lesssim \|f\|_{\dot{W}^{\gamma,q}_c},
\]
for all $f \in C^\infty_0(\R^d\backslash \{0\})$.
\item[2.] If $1<q<\infty$ satisfies $\max\left\{\frac{\gamma}{d}, \frac{\rho}{d} \right\} <\frac{1}{q} <\min\left\{1, \frac{d-\rho}{d} \right\}$, then
\[
\|f\|_{\dot{W}^{\gamma, q}_c} \lesssim \|f\|_{\dot{W}^{\gamma,q}},
\]
for all $f \in C^\infty_0(\R^d\backslash \{0\})$.
\end{itemize}
\end{prop}
\begin{rem} \label{rem equivalent sobolev norms}
\begin{itemize}
\item[1.] When $c>0$, we have $\rho<0$. Therefore, $\|u\|_{\dot{W}^{\gamma, q}}$ is equivalent to $\|u\|_{\dot{W}^{\gamma, q}_c}$ provided that $0<\gamma<2$ and
\begin{align}
\frac{\gamma}{d}<\frac{1}{q}<1 \quad \text{or} \quad 1<q<\frac{d}{\gamma}. \label{equivalent condition positive c}
\end{align}
\item[2.] When $-\lambda(d)\leq c<0$, we have $0<\rho<\frac{d-2}{2}$. Thus $\|u\|_{\dot{W}^{\gamma, q}} \sim \|u\|_{\dot{W}^{\gamma, q}_c}$ provided that $0<\gamma<2$ and
\begin{align}
\frac{\gamma+\rho}{d}<\frac{1}{q}<\frac{d-\rho}{d} \quad \text{or} \quad \frac{d}{d-\rho}<q<\frac{d}{\gamma+\rho}. \label{equivalent condition negative c}
\end{align} 
\end{itemize}
\end{rem}
We next recall the fractional derivative estimates due to Christ-Weinstein \cite{ChristWeinstein}. The equivalence of Sobolev spaces given in Proposition $\ref{prop equivalence sobolev spaces}$ allows us to use the same estimates for powers of $P_c$ with a certain set of exponents.  
\begin{lem}[Fractional derivative estimates] \label{lem fractional calculus}
\begin{itemize}
\item[1.] Let $\gamma \geq 0, 1<r<\infty$ and $1<p_1, q_1, p_2, q_2\leq \infty$ satisfying $
\frac{1}{r}=\frac{1}{p_1} + \frac{1}{q_1} = \frac{1}{p_2} + \frac{1}{q_2}$. Then
\[
\||\nabla|^\gamma (fg)\|_{L^r} \lesssim \|f\|_{L^{p_1}} \||\nabla|^\gamma g\|_{L^{q_1}} + \||\nabla|^\gamma f\|_{L^{p_2}} \|g\|_{L^{q_2}}. 
\]
\item[2.] Let $G \in C^1(\C), \gamma \in (0,1], 1<r, q<\infty$ and $1<p \leq \infty$ satisfying $\frac{1}{r} = \frac{1}{p} +\frac{1}{q}$. Then 
\[
\||\nabla|^\gamma G(f)\|_{L^r} \lesssim \|G'(f)\|_{L^p} \||\nabla|^\gamma f\|_{L^q}.
\]
\end{itemize}
\end{lem}
\subsection{Convergences of operators}
In this subsection, we recall the convergence of operators of \cite{KillipMiaoVisanZhangZheng-energy} arising from the fact that $P_c$ does not commute with translations. 
\begin{defi}
Suppose $(x_n)_{n\in \N} \subset \R^d$. We define
\begin{align}
\renewcommand*{\arraystretch}{1.2}
P^n_c:=-\Delta + \frac{c}{|x+x_n|^2}, \quad P^\infty_c:= \left\{
\begin{array}{cl}
-\Delta +\frac{c}{|x+x_\infty|^2} &\text{if } x_n \rightarrow x_\infty \in \R^d, \\
-\Delta &\text{if } |x_n|\rightarrow \infty.
\end{array}
\right. \label{definition operators}
\end{align}
\end{defi}
By definition, we have $P_c[f(x-x_n)] = [P^n_c f](x-x_n)$. The operator $P^\infty_c$ appears as a limit of the operators $P^n_c$ in the following senses:
\begin{lem}[Convergence of operators \cite{KillipMiaoVisanZhangZheng-energy}] \label{lem convergence operators}
Let $d\geq 3$ and $c\ne 0$ be such that $c>-\lambda(d)$. Suppose $(t_n)_{n\in\N} \subset \R$ satisfies $t_n \rightarrow t_\infty\in \R$, and $(x_n)_{n\in \N} \subset \R^d$ satisfies $x_n\rightarrow x_\infty \in \R^d$ or $|x_n|\rightarrow \infty$. Then,
\begin{align}
\lim_{n\rightarrow \infty} & \|P^n_c f - P^\infty_c f\|_{\dot{H}^{-1}} =0, \quad \text{for all} \quad f \in \dot{H}^1, \label{convergence in H minus 1} \\
\lim_{n\rightarrow \infty} & \|e^{-it_nP^n_c} f - e^{-it_\infty P^\infty_c} f\|_{\dot{H}^{-1}} =0, \quad \text{for all} \quad f \in \dot{H}^{-1}, \label{convergence evolution in H minus 1} \\
\lim_{n\rightarrow \infty} & \|\sqrt{P^n_c} f - \sqrt{P^\infty_c} f\|_{L^2} =0, \quad \text{for all} \quad f \in \dot{H}^1. \label{convergence in L2}
\end{align}
Furthermore, for any $(p,q) \in S$ with $p \ne 2$,
\begin{align}
\lim_{n\rightarrow \infty} \|e^{-itP^n_c} f - e^{-it P^\infty_c} f\|_{L^p(\R, L^q)} =0, \quad \text{for all} \quad f \in L^2. \label{convergence in strichart norm}
\end{align}
\end{lem}
We refer the reader to \cite[Lemma 3.3]{KillipMiaoVisanZhangZheng-energy} for the proof of Lemma $\ref{lem convergence operators}$. 
\section{Local well-posedness}
In this section, we study the local well-posedness for the $(\nlsct)$ in the energy-subcritical and energy-critical cases. To our knowledge, there are two possible ways to show the local well-posedness in $H^1$ for the classical nonlinear Schr\"odinger equation (NLS$_0$): the Kato's method and the energy method. The Kato's method is based on the contraction mapping principle using Strichartz estimates. This method is very effective to study the (NLS$_0$) in general Sobolev spaces. The energy method, on the other hand, does not use Strichartz estimates and only allows to prove the existence of solutions in the energy space. But, on one hand, it provides a useful tool to study the (NLS$_0$) in a general domain $\Omega$ where Strichartz estimates are not available in general. We refer the reader to \cite{Cazenave} for more details. In the presence of the singular potential $c|x|^{-2}$, even though Strichartz estimates are available (see \cite{BurqPlanchonStalkerTahvildar-Zadeh,BoucletMizutani}), the Kato's method does not allow to study the $(\nlsct)$ in the energy space with the full range $c>-\lambda(d)$. The reason for this is that the homogeneous Sobolev spaces $\dot{W}^{\gamma, q}_c$ and the usual ones $\dot{W}^{\gamma, q}$ are equivalent only in a certain range of $\gamma$ and $q$ (see Subsection $\ref{subsection equivalent sobolev norms}$). Moreover, Okazawa-Suzuki-Yokota in \cite{OkazawaSuzukiYokota-energy} pointed out that the energy method developed by Cazenave is not enough to study the $(\nlsct)$ in the energy space. They thus formulated an improved energy method to treat the equation. More precisely, they proved the following:
\begin{theorem}[\cite{OkazawaSuzukiYokota-energy}] \label{theorem energy method Okazawa-Suzuki-Yokota}
Let $d\geq 3, c >-\lambda(d)$. Then the $(\nlsce)$ is well posed in $H^1$:
\begin{itemize}
\item locally if $0 \leq \alpha <\alpha^\star$,
\item globally if $0\leq \alpha <\alpha_\star$.
\end{itemize}
Here $\alpha_\star, \alpha^\star$ are given in $(\ref{define alpha star})$. 
\end{theorem}
We refer the reader to \cite[Theorem 5.1]{OkazawaSuzukiYokota-energy} for the proof of this result.
\begin{rem} \label{rem energy method Okazawa-Suzuki-Yokota}
\begin{itemize}
\item[1.] The energy method developed by Okazawa-Suzuki-Yokota is only available for the energy-subcritical case (i.e. $\alpha<\alpha^\star$) and not for the energy-critical case $\alpha=\alpha^\star$. The last case should rely on Kato's method (see Proposition $\ref{prop local wellposedness energy-critical inverse square}$ below).
\item[2.] Theorem $\ref{theorem energy method Okazawa-Suzuki-Yokota}$ tells us that $H^1$ blowup solutions may occur only on $\alpha_\star \leq \alpha \leq \alpha^\star$. 
\item[3.] The same well-posedness for the $(\nlsct)$ as in Theorem $\ref{theorem energy method Okazawa-Suzuki-Yokota}$ holds true when one replaces $\R^d$ by a bounded domain $\Omega$ (see again \cite{OkazawaSuzukiYokota-energy}). In this consideration, Suzuki in \cite{Suzuki-bounded} proved a similar result for the $(\nlsct)$ on $\Omega$ with $c=\lambda(d)$. 
\end{itemize}
\end{rem}
We now consider the energy-critical case $\alpha=\alpha^\star$. 
\begin{prop} \label{prop local wellposedness energy-critical inverse square}
Let $d\geq 3, c>-\frac{d^2+4d}{(d+2)^2} \lambda(d)$ and $\alpha=\alpha^\star$. Then for every $u_0 \in H^1$, there exist $T_*, T^*\in (0, \infty]$ and a unique strong $H^1$ solution to the $(\nlsce)$ defined  on the maximal interval $(-T_*, T^*)$. Moreover, if $\|u_0\|_{\dot{H}^1}<\ep$ for some $\ep>0$ small enough, then $T_* = T^*=\infty$ and the solution is scattering in $H^1$, i.e. there exist $u_0^\pm \in H^1$ such that
\[
\lim_{t\rightarrow \pm \infty} \|u(t) - e^{-itP_c} u^\pm_0\|_{H^1} = 0.
\] 
\end{prop}
Before giving the proof of this result, let us introduce some notations. In this section, we denote
\[
p = \frac{2(d+2)}{d-2}, \quad q=\frac{2d(d+2)}{d^2+4}.
\]
It is easy to check that $(p,q)$ is a Schr\"odinger admissible pair and 
\[
\frac{1}{p}=\frac{1}{q}-\frac{1}{d}.
\]
The last equality allows us to use the Sobolev embedding $\dot{W}^{1,q} \subset L^p$. 
Moreover, in the view of $(\ref{equivalent condition positive c})$ and $(\ref{equivalent condition negative c})$, it is easy to check that $\dot{W}^{1,q}_c$ is equivalent to $W^{1,q}$ provided that $c>-\frac{d^2+4d}{(d+2)^2} \lambda(d)$.\newline
\textit{Proof of Proposition $\ref{prop local wellposedness energy-critical inverse square}$.}
We only consider the positive time, the negative time is similar. Let us define
\[
X:= \Big\{ u \in C(I, H^1) \cap L^p(I, W^{1,q}) \ \Big| \ \|u\|_{L^p(I, \dot{W}^{1,q})} \leq M  \Big\}
\]
equipped with the distance
\[
d(u,v):= \|u-v\|_{L^p(I, L^q)},
\]
where $I=[0,T]$ with $T, M>0$ to be chosen later. 
By the Duhamel formula, it suffices to prove that the functional
\[
\Phi(u)(t)= e^{-itP_c} u_0 +i\int_0^t e^{-i(t-s)P_c} |u(s)|^{\alpha^\star} u(s)ds =: u_{\text{hom}}(t)+ u_{\text{inh}}(t)
\]
is a contraction on $(X,d)$. Using Strichartz estimates and the fact $\|u\|_{\dot{W}^{1,q}_c} \sim \|u\|_{\dot{W}^{1,q}}$, we have
\[
\|u_{\text{hom}}\|_{L^p(I, \dot{W}^{1,q})}\sim \|u_{\text{hom}}\|_{L^p(I, \dot{W}^{1,q}_c)} \lesssim \|u_0\|_{\dot{H}^1_c}\sim \|u_0\|_{\dot{H}^1}.
\]
This shows that $\|u_{\text{hom}}\|_{L^p(I, \dot{W}^{1,q})} \leq \ep$ for some $\ep>0$ small enough provided that $T$ is small or $\|u_0\|_{\dot{H}^1}$ is small. By Strichartz estimates, the equivalence $\|u\|_{\dot{W}^{1,q}_c} \sim \|u\|_{\dot{W}^{1,q}}$, the fractional derivative estimates and the Sobolev embedding $\dot{W}^{1,q} \subset L^p$,
\begin{align*}
\|u_{\text{inh}}\|_{L^p(I, \dot{W}^{1,q}_c)}\sim \|u_{\text{inh}}\|_{L^p(I, \dot{W}^{1,q}_c)} &\lesssim \||u|^{\alpha^\star} u\|_{L^2(I, \dot{W}^{1,\frac{2d}{d+2}}_c)} \sim \||u|^{\alpha^\star} u\|_{L^2(I, \dot{W}^{1,\frac{2d}{d+2}})} \\
&\lesssim \|u\|^{\alpha^\star}_{L^p(I, L^p)} \|u\|_{L^p(I, \dot{W}^{1,q})} \lesssim \|u\|^{\alpha^\star+1}_{L^p(I, \dot{W}^{1,q})}. 
\end{align*}
Note that it is easy to check that $\dot{W}^{1,\frac{2d}{d+2}}_c \sim \dot{W}^{1,\frac{2d}{d+2}}$. Similarly,
\begin{align*}
\||u|^{\alpha^\star} u - |v|^{\alpha^\star}v \|_{L^2(I, L^{\frac{2d}{d+2}})} &\lesssim \Big(\|u\|^{\alpha^\star}_{L^p(I, L^p)} + \|v\|^{\alpha^\star}_{L^p(I, L^p)} \Big) \|u-v\|_{L^p(I, L^q)} \\
&\lesssim \Big(\|u\|^{\alpha^\star}_{L^p(I, \dot{W}^{1,q})} + \|v\|^{\alpha^\star}_{L^p(I, \dot{W}^{1,q})} \Big) \|u-v\|_{L^p(I, L^q)}.
\end{align*}
This implies that for any $u,v \in X$, there exists $C>0$ independent of $T$ and $u_0 \in H^1$ such that
\begin{align*}
\|\Phi(u)\|_{L^p(I, \dot{W}^{1,q}_c)} &\leq \ep + C M^{\alpha^\star+1}, \\
d(\Phi(u), \Phi(v)) &\leq CM^{\alpha^\star} d(u,v).
\end{align*}
If we choose $\ep$ and $M$ small so that
\[
CM^{\alpha^\star} \leq \frac{1}{2}, \quad \ep + \frac{M}{2} \leq M,
\]
then $\Phi$ is a contraction on $(X,d)$. This shows the local existence. It remains to show the scattering for small data. As mentioned above, when $\|u_0\|_{\dot{H}^1}$ is small enough, we can take $T^*=\infty$. By Strichartz estimates, we have for $0<t_1<t_2$,
\begin{align*}
\|e^{it_2 P_c} u(t_2) - e^{it_1 P_c} u(t_1)\|_{\dot{H}^1} \sim \|e^{it_2 P_c} u(t_2) - e^{it_1 P_c} u(t_1)\|_{\dot{H}^1_c} &=\Big\| - i \int_{t_1}^{t_2} e^{isP_c} |u(s)|^{\alpha^\star} u(s) ds \Big\|_{\dot{H}^1_c} \\
&\lesssim \||u|^{\alpha^\star} u\|_{L^2([t_1, t_2], \dot{W}^{1,\frac{2d}{d+2}}_c)}\\
&\sim \||u|^{\alpha^\star} u\|_{L^2([t_1, t_2], \dot{W}^{1,\frac{2d}{d+2}})}\\
&\lesssim \|u\|^{\alpha^\star+1}_{L^p([t_1,t_2], \dot{W}^{1,q})}.
\end{align*}
Similarly, 
\[
\|e^{it_2 P_c} u(t_2) - e^{it_1 P_c} u(t_1)\|_{L^2} \lesssim \||u|^{\alpha^\star} u\|_{L^2([t_1,t_2], L^{\frac{2d}{d+2}})} \lesssim \|u\|^{\alpha^\star}_{L^p([t_1,t_2], \dot{W}^{1,q})} \|u\|_{L^p([t_1,t_2], L^q)}.
\]
This shows that
\[
\|e^{it_2P_c} u(t_2) - e^{it_1 P_c} u(t_1)\|_{H^1} \rightarrow 0,
\]
as $t_1, t_2 \rightarrow +\infty$. Thus the limit $u_0^+: \lim_{t\rightarrow +\infty} e^{itP_c} u(t)$ exists in $H^1$. Moreover,
\[
u(t) - e^{-itP_c} u_0^+ = -i \int_t^{+\infty} e^{-i(t-s)P_c} |u(s)|^{\alpha^\star} u(s) ds.
\]
Estimating as above, we get
\[
\lim_{t\rightarrow +\infty} \|u(t) - e^{-itP_c} u^+_0\|_{H^1} =0.
\]
The proof is complete.
\defendproof
\section{Variational analysis} \label{section variational analysis}
\setcounter{equation}{0}
In this section, we recall the sharp Gagliardo-Nirenberg and the sharp Sobolev embedding inequalities related to the $(\nlsct)$. \newline
\indent Let us start with the following sharp Gagliardo-Nirenberg inequality:
\begin{align}
\|f\|^{\alpha+2}_{L^{\alpha+2}} \leq C_{\text{GN}}(c) \|f\|_{L^2}^{\frac{4-(d-2)\alpha}{2}} \|f\|_{\dot{H}^1_c}^{\frac{d\alpha}{2}}. \label{sharp gagliardo nirenberg inequality}
\end{align}
The sharp constant $C_{\text{GN}}(c)$ is defined by
\[
C_{\text{GN}}(c) := \sup\left\{ J_c(f) : f\in H^1_c\backslash\{0\}\right\},
\]
where $J_c(f)$ is the Weinstein functional
\[
J_c(f):= \|f\|^{\alpha+2}_{L^{\alpha+2}} \div \Big[\|f\|_{L^2}^{\frac{4-(d-2)\alpha}{2}} \|f\|^{\frac{d\alpha}{2}}_{\dot{H}^1_c}\Big].
\]
We also consider the sharp radial Gagliardo-Nirenberg inequality:
\begin{align}
\|f\|^{\alpha+2}_{L^{\alpha+2}} \leq C_{\text{GN}}(c, \text{rad}) \|f\|_{L^2}^{\frac{4-(d-2)\alpha}{2}} \|f\|_{\dot{H}^1_c}^{\frac{d\alpha}{2}}, \quad f \text{ radial}, \label{sharp gagliardo nirenberg inequality radial}
\end{align}
where the sharp constant $C_{\text{GN}}(c,\text{rad})$ is defined by
\[
C_{\text{GN}}(c,\text{rad}) := \sup\left\{ J_c(f) : f\in H^1_c\backslash\{0\}, f \text{ radial}\right\}.
\]
When $c=0$, Weinstein in \cite{Weinstein} proved that the sharp constant $C_{\text{GN}}(0)$ is attained by the function $Q_0$, which is the unique positive radial solution of 
\begin{align}
\Delta Q_0 -Q_0 + Q_0^{\alpha+1} =0. \label{ground state equation 0}
\end{align}
Recently, Killip-Murphy-Visan-Zheng extended Weinstein's result to $c\ne 0$. More precisely, we have the following:
\begin{theorem}[Sharp Gagliardo-Nirenberg inequality \cite{KillipMurphyVisanZheng}] \label{theorem gagliardo nirenberg inequality}
Let $d\geq 3, 0<\alpha<\alpha^\star$ and $c\ne 0$ be such that $c>-\lambda(d)$. Then we have $C_{\emph{GN}}(c) \in (0,\infty)$ and 
\begin{itemize}
\item[1.] if $-\lambda(d)<c<0$, then the equality in $(\ref{sharp gagliardo nirenberg inequality})$ is attained by a function $Q_c \in H^1_c$, which is a non-zero, non-negative, radial solution to the elliptic equation
\begin{align}
- P_c Q_c -Q_c + Q_c^{\alpha+1} =0. \label{ground state equation}
\end{align}
\item[2.] if $c>0$, then $C_{\emph{GN}}(c) = C_{\emph{GN}}(0)$ and the equality in $(\ref{sharp gagliardo nirenberg inequality})$ is never attained. However, the constant $C_{\emph{GN}}(c,\emph{rad})$ is attained by a function $Q_{c,\emph{rad}}$ which is a solution to the elliptic equation
\begin{align}
- P_c Q_{c, \emph{rad}} -Q_{c, \emph{rad}} + Q_{c, \emph{rad}}^{\alpha+1} =0. \label{ground state equation radial}
\end{align}
\end{itemize}
\end{theorem} 
\begin{proof}
In \cite[Theorem 3.1]{KillipMurphyVisanZheng}, Killip-Murphy-Visan-Zheng gave the proof for $d=3$ and $\alpha=2$. For reader's convenience, we provide some details for the general case. Since $\|f\|_{\dot{H}^1_c} \sim \|f\|_{\dot{H}^1}$, we see that $J_c(f) \sim J_0(f)$. Thus the standard Gagliardo-Nirenberg inequality (i.e. $(\ref{sharp gagliardo nirenberg inequality})$ with $c=0$) implies $0<C_{\text{GN}}(c)<\infty$. \newline
\indent Let us consider the case $-\lambda(d)<c<0$. Let $(f_n)_n \subset H^1_c \backslash \{0\}$ be a maximizing sequence, i.e. $J_c(f_n) \nearrow C_{\text{GN}}(c)$. Let $f^*_n$ be the Schwarz symmetrization of $f_n$ (see e.g. \cite{LiebLoss}). Using the fact that the Schwarz symmetrization preserves $L^q$ norm and does not increase $\dot{H}^1$ norm together with the Riesz rearrangement inequality
\begin{align}
\int c|x|^{-2} |f^*(x)|^2  \leq \int c|x|^{-2}|f(x)|^2 dx,  \label{riesz rearrangement inequality}
\end{align}
for $c<0$, we see that $J_c(f_n) \leq J_c(f^*_n)$. Thus we may assume that each $f_n$ is radial. Note that $(\ref{riesz rearrangement inequality})$ plays an important role in order to restore the lack of compactness due to translations. We next observe that the functional $J_c$ is invariant under the scaling 
\[
f_{\lambda, \mu} (x) := \lambda f(\mu x), \quad \lambda, \mu>0.
\]
Indeed, a simple computation shows
\[
\|f_{\lambda, \mu}\|^2_{\dot{H}^1_c} = \lambda^2 \mu^{2-d} \|f\|^2_{\dot{H}^1_c}, \quad \|f_{\lambda, \mu}\|^2_{L^2} = \lambda^2 \mu^{-d} \|f\|_{L^2}^2, \quad \|f_{\lambda, \mu}\|^{\alpha+2}_{L^{\alpha+2}} = \lambda^{\alpha+2} \mu^{-d} \|f\|^{\alpha+2}_{L^{\alpha+2}}.
\]
We thus get $J_c(f_{\lambda, \mu}) = J_c(f)$. We now rescale the sequence $(f_n)_n$ by setting $g_n(x):= \lambda_n f_n(\mu_n x)$, where
\[
\lambda_n = \frac{\|f_n\|^{d/2-1}_{L^2}}{\|f_n\|^{d/2}_{\dot{H}^1_c}}, \quad \mu_n = \frac{\|f_n\|_{L^2}}{\|f_n\|_{\dot{H}^1_c}}.
\]
It is easy to see that $\|g_n\|_{L^2} = \|g_n\|_{\dot{H}^1_c} = 1$. We thus get a maximizing sequence $(g_n)_n$ of $J_c$, which is bounded in $H^1_c$.  We have from the compactness lemma (see e.g. \cite{Weinstein}) that $H^1_{\text{rad}}(\R^d) \hookrightarrow L^{\alpha+2}(\R^d)$ compactly for any $0<\alpha<\frac{4}{d-2}$. Therefore, there exists $g \in H^1_c$ such that, up to a subsequence, $g_n \rightarrow g$ strongly in $L^{\alpha+2}$ as well as weakly in $H^1_c$. By the weak convergence, $\|g\|_{L^2} \leq 1$ and $\|g\|_{\dot{H}^1_c} \leq 1$. Hence,
\[
C_{\text{GN}}(c) = \lim_{n\rightarrow \infty} J_c(g_n) =\|g\|^{\alpha+2}_{L^{\alpha+2}} \leq J_c(g) \leq C_{\text{GN}}(c).
\]
Thus, we have $J_c(g)=\|g\|^{\alpha+2}_{L^{\alpha+2}} = C_{\text{GN}}(c)$ and $\|g\|_{L^2} = \|g\|_{\dot{H}^1_c}=1$. Therefore, $g$ is a maximizer for the Weinstein functional $J_c$, and so $g$ must satisfy the Euler-Lagrange equation
\[
\frac{d}{d\ep}\Big|_{\ep=0} J_c(g+\ep h) =0, \quad \forall h \in C^\infty_0 (\R^d\backslash \{0\}).
\] 
Taking into consideration that $\|g\|_{L^2}= \|g\|_{\dot{H}^1_c} =1$ and $C_{\text{GN}}(c)=\|g\|^{\alpha+2}_{L^{\alpha+2}}$, we get
\[
-\frac{d\alpha}{2} C_{\text{GN}}(c) P_c g - \frac{4-(d-2)\alpha}{2} C_{\text{GN}}(c) g + (\alpha+2) g^{\alpha+1} =0. 
\]
If we define $Q_c$ by $g(x) = \lambda Q_c(\mu x)$ with
\[
\lambda = \sqrt[\alpha]{\frac{4-(d-2)\alpha}{2(\alpha+2)} C_{\text{GN}}(c)}, \quad \mu = \sqrt{\frac{4-(d-2)\alpha}{d\alpha}},
\]
then $Q_c$ solves $(\ref{ground state equation})$. This proves Item 1. \newline
\indent In the case $c>0$, we consider a sequence $(x_n)_n \subset \R^d$ with $|x_n|\rightarrow \infty$. Let $Q_0$ be the unique positive radial solution to $(\ref{ground state equation 0})$. Using the definition $(\ref{definition operators})$ and $(\ref{convergence in L2})$, we have
\[
\|Q_0(\cdot-x_n)\|^2_{\dot{H}^1_c} = \| \sqrt{P_c}[Q_0(\cdot-x_n)]\|^2_{L^2} = \|[\sqrt{P^n_c} Q_0](\cdot-x_n)\|^2_{L^2} \rightarrow \|\sqrt{P^\infty_c} Q_0\|^2_{L^2}= \|Q_0\|^2_{\dot{H}^1}.
\]
We thus get
\[
J_c(Q_0(\cdot-x_n)) \rightarrow J_0 (Q_0) = C_{\text{GN}}(0),
\]
hence $C_{\text{GN}}(0) \leq C_{\text{GN}}(c)$. Since $c>0$, it is obvious that $\|f\|_{\dot{H}^1_x} < \|f\|_{\dot{H}^1_c}$ for any $f \in H^1 \backslash\{0\}$. The sharp Gagliardo-Nirenberg inequality for $c=0$ then implies
\[
\|f\|^{\alpha+2}_{L^{\alpha+2}} \leq C_{\text{GN}}(0)\|f\|^{\frac{4-(d-2)\alpha}{2}}_{L^2} \|f\|^{\frac{d\alpha}{2}}_{\dot{H}^1} < C_{\text{GN}}(0)\|f\|^{\frac{4-(d-2)\alpha}{2}}_{L^2} \|f\|^{\frac{d\alpha}{2}}_{\dot{H}^1_c},
\]
whence $J_c(f) < C_{\text{GN}}(0)$ for any $f \in H^1 \backslash\{0\}$. Since $H^1$ is equivalent to $H^1_c$, we obtain $C_{\text{GN}}(c)<C_{\text{GN}}(0)$. Therefore, $C_{\text{GN}}(c)=C_{\text{GN}}(0)$. The last estimate also shows that the equality in $(\ref{sharp gagliardo nirenberg inequality})$ is never attained. Note also that the estimate $(\ref{riesz rearrangement inequality})$ fails to hold true when $c>0$. If we only consider radial functions, then the result follows exactly as the case $-\lambda(d)<c<0$ (after passing to the Schwarz symmetrization sequence). The proof is complete.
\end{proof}
\begin{rem} \label{rem ground state equation}
\begin{itemize}
\item[1.] When $-\lambda(d)<c<0$, the proof of Theorem $\ref{theorem gagliardo nirenberg inequality}$ shows that there exist solutions to the elliptic equation $(\ref{ground state equation})$, which are non-zero, non-negative and radially symmetric. However, unlike the standard case $c=0$, we do not know that the uniqueness (up to symmetries) of these solutions. Moreover, any positive maximiser of $J_c$ is radial. Furthermore, if $Q_c$ is a maximiser of $J_c$, then by multiplying $(\ref{ground state equation})$ with $Q_c$ and $x\cdot \nabla Q_c$ and integrating over $\R^d$, we obtain the following Pohozaev identities:
\[
\|Q_c\|^2_{\dot{H}^1_c} + \|Q_c\|^2_{L^2} -\|Q_c\|^{\alpha+2}_{L^{\alpha+2}} = \frac{d-2}{2} \|Q_c\|^2_{\dot{H}^1_c} + \frac{d}{2}\|Q_c\|^2_{L^2} -\frac{d}{\alpha+2}\|Q_c\|^{\alpha+2}_{L^{\alpha+2}}=0.
\]
In particular, 
\begin{align}
\|Q_c\|^2_{L^2} = \frac{4-(d-2)\alpha}{d\alpha} \|Q_c\|^2_{\dot{H}^1_c}= \frac{4-(d-2)\alpha}{2(\alpha+2)} \|Q_c\|^{\alpha+2}_{L^{\alpha+2}}, \label{ground state property}
\end{align}
and
\begin{align}
C_{\text{GN}}(c) &= \frac{2(\alpha+2)}{4-(d-2)\alpha}\Big[\frac{4-(d-2)\alpha}{d\alpha}\Big]^{\frac{d\alpha}{4}} \frac{1}{\|Q_c\|^\alpha_{L^2}} \label{pohozaev identity} \\
&= \frac{2(\alpha+2)}{d\alpha}\Big[\frac{d\alpha}{4-(d-2)\alpha}\Big]^{\frac{4-(d-2)\alpha}{4}} \frac{1}{\|Q_c\|^\alpha_{\dot{H}^1_c}} \nonumber \\
&= \frac{[2(\alpha+2)]^{\frac{\alpha+2}{2}}}{[4-(d-2)\alpha]^{\frac{4-(d-2)\alpha}{4}} [d\alpha]^{\frac{d\alpha}{4}}} \frac{1}{\|Q_c\|^{\frac{\alpha(\alpha+2)}{2}}_{L^{\alpha+2}}}. \nonumber
\end{align}
In particular, all maximizers of $J_c$ have the same $L^2, \dot{H}^1_c, L^{\alpha+2}$-norms. We also have
\begin{align}
E_c(Q_c) = \frac{d\alpha-4}{2[4-(d-2)\alpha]} \|Q_c\|^2_{L^2} = \frac{d\alpha-4}{2d\alpha}\|Q_c\|^2_{\dot{H}^1_c}. \label{energy ground state}
\end{align}
In particular, in the mass-critical case, i.e. $\alpha=\frac{4}{d}$, we have $E_c(Q_c)=0$. 
\item[2.] Since the identities $(\ref{ground state property})-(\ref{energy ground state})$ hold true for $c=0$, we have from Theorem $\ref{theorem gagliardo nirenberg inequality}$ that for any $c>-\lambda(d)$,
\begin{align}
C_{\text{GN}}(c) &= \frac{2(\alpha+2)}{4-(d-2)\alpha}\Big[\frac{4-(d-2)\alpha}{d\alpha}\Big]^{\frac{d\alpha}{4}} \frac{1}{\|Q_{\cbar}\|^\alpha_{L^2}} \label{sharp gagliardo nirenberg constant} \\
&= \frac{2(\alpha+2)}{d\alpha}\Big[\frac{d\alpha}{4-(d-2)\alpha}\Big]^{\frac{4-(d-2)\alpha}{4}} \frac{1}{\|Q_{\cbar}\|^\alpha_{\dot{H}^1_{\cbar}}} \nonumber \\
&= \frac{[2(\alpha+2)]^{\frac{\alpha+2}{2}}}{[4-(d-2)\alpha]^{\frac{4-(d-2)\alpha}{4}} [d\alpha]^{\frac{d\alpha}{4}}} \frac{1}{\|Q_{\cbar}\|^{\frac{\alpha(\alpha+2)}{2}}_{L^{\alpha+2}}}, \nonumber
\end{align}
where $\cbar = \min \{c,0\}$.
\item[3.] Let $H(c)$ and $K(c)$ be as in $(\ref{define Hc Kc})$. Using $(\ref{ground state property}), (\ref{pohozaev identity})$ and $(\ref{energy ground state})$, it is easy to see that
\begin{align}
H(c) = 
\frac{d\alpha-4}{2d\alpha}\Big[\frac{d\alpha}{2(\alpha+2)} C_{\text{GN}}(c) \Big]^{-\frac{4}{d\alpha-4}}, \label{relation energy}
\end{align}
and
\begin{align}
K(c)= 
\Big[\frac{d\alpha}{2(\alpha+2)} C_{\text{GN}}(c) \Big]^{-\frac{2}{d\alpha-4}}. \label{relation kinetic}
\end{align}
In particular, 
\begin{align}
H(c)=\frac{d\alpha-4}{2d\alpha} K(c)^2. \label{relation energy kinetic}
\end{align}
\item[4.] When $c>0$, we see that the same identities as in $(\ref{ground state property})$, $(\ref{pohozaev identity})$, $(\ref{energy ground state})$, $(\ref{sharp gagliardo nirenberg constant})$, $(\ref{energy ground state})$, $(\ref{relation energy})$, $(\ref{relation kinetic})$ and $(\ref{relation energy kinetic})$ hold true with $Q_{c,\text{rad}}$, $C_{\text{GN}}(c,\text{rad})$, $H(c,\text{rad})$ and $K(c,\text{rad})$ in place of $Q_{c}$, $C_{\text{GN}}(c)$, $H(c)$ and $K(c)$ respectively.
\end{itemize}
\end{rem}
\indent Let us now consider the sharp Sobolev embedding inequality:
\begin{align}
\|f\|_{L^{\alpha^\star+2}} \leq C_{\text{SE}}(c)\|f\|_{\dot{H}^1_c}, \label{sharp sobolev embedding}
\end{align}
where the sharp constant $C_{\text{SE}}(c)$ is defined by
\[
C_{\text{SE}}(c):= \sup \left\{\|f\|_{L^{\alpha^\star+2}} \div \|f\|_{\dot{H}^1_c} \ : \ f \in \dot{H}^1_c \backslash\{0\} \right\}.
\]
We also consider the sharp radial Sobolev embedding inequality
\begin{align}
\|f\|_{L^{\alpha^\star+2}} \leq C_{\text{SE}}(c, \text{rad})\|f\|_{\dot{H}^1_c}, \quad f \text{ radial}
\label{sharp sobolev embedding radial}
\end{align}
where the sharp constant $C_{\text{SE}}(c, \text{rad})$ is defined by
\[
C_{\text{SE}}(c, \text{rad}):= \sup\left\{\|f\|_{L^{\alpha^\star+2}} \div \|f\|_{\dot{H}^1_c} \ : \ f \in \dot{H}^1_c \backslash\{0\}, f \text{ radial} \right\}.
\]
\indent When $c=0$, it was proved by Aubin \cite{Aubin} and Talenti \cite{Talenti} that the constant $C_{\text{SE}}(0)$ is attained by functions $f(x)$ of a form $\lambda W_0(\mu x+y)$ for some $\lambda \in \C, \mu>0$ and $y \in \R^d$, where $W_0$ is given in $(\ref{define W_0})$. \newline
\indent When $c\ne 0$, Killip-Miao-Visan-Zhang-Zheng in \cite{KillipMiaoVisanZhangZheng-energy} proved the following result.
\begin{theorem}[Sharp Sobolev embedding inequality \cite{KillipMiaoVisanZhangZheng-energy}]  \label{theorem sobolev embedding inequality}
Let $d\geq 3$ and $c\ne 0$ be such that $c>-\lambda(d)$. Then $C_{\emph{SE}}(c) \in (0,\infty)$ and 
\begin{itemize}
\item[1.] if $-\lambda(d)<c<0$, then the equality in $(\ref{sharp sobolev embedding})$ is attained by functions $f(x)$ of the form $\lambda W_c(\mu x)$ for some $\lambda \in \C$ and some $\mu>0$, where $W_c$ is given in $(\ref{define W_c})$.
\item[2.] if $c>0$, then $C_{\emph{SE}}(c)=C_{\emph{SE}}(0)$ and the equality in $(\ref{sharp sobolev embedding})$ is never attained. However, $C_{\emph{SE}}(c,\emph{rad})$ is attained by functions $f(x)$ of the form $\lambda W_c(\mu x)$ for some $\lambda \in \C$ and some $\mu>0$, where $W_c$ is again given in $(\ref{define W_c})$.
\end{itemize}
\end{theorem}
We refer the reader to \cite[Proposition 7.2]{KillipMiaoVisanZhangZheng-energy} for the proof of this result. Note that the non-existence of optimizers to the Sobolev embedding inequality for $c>0$ is a consequence of the failure of compactness due to translation. If we restrict our consideration to radial functions, the compactness is restored. To end this section, we recall some properties related to $W_c$ (see \cite[Section 7]{KillipMiaoVisanZhangZheng-energy} for more details). It is not difficult to verify that $W_c$ solves the elliptic equation
\[
P_c W_c = |W_c|^{\alpha^\star} W_c.
\] 
This implies in particular
\begin{align}
\|W_c\|^2_{\dot{H}^1_c} = \|W_c\|^{\alpha^\star+2}_{L^{\alpha^\star+2}}. \label{property W_c}
\end{align}
Combining with Theorem $\ref{theorem sobolev embedding inequality}$, we have for $-\lambda(d)<c<0$,
\begin{align}
\|W_c\|^2_{\dot{H}^1_c} &= \|W_c\|^{\alpha^\star+2}_{L^{\alpha^\star+2}} = C_{\text{SE}}(c)^{-d}, \label{property W_c 1}\\
E_c(W_c) &= \frac{1}{2}\|W_c\|^2_{\dot{H}^1_c} -\frac{1}{\alpha^\star+2} \|W_c\|^{\alpha^\star+2}_{L^{\alpha^\star+2}} = d^{-1}C_{\text{SE}}(c)^{-d}. \label{property W_c 2}
\end{align}
Note that $(\ref{property W_c 1})$ and $(\ref{property W_c 2})$ hold true for $c=0$. In particular, we have for any $c\ne 0$ satisfying $c>-\lambda(d)$,
\begin{align}
C_{\text{SE}}(c)=\|W_{\cbar}\|^{-\frac{2}{d}}_{\dot{H}^1_{\cbar}} = \|W_{\cbar}\|^{-\frac{\alpha^\star+2}{d}}_{L^{\alpha^\star+2}} = [dE_{\cbar}(W_{\cbar})]^{-\frac{1}{d}}. \label{relation sharp sobolev embedding constant}
\end{align}
Similarly, we have for $c>0$ that
\begin{align}
\|W_c\|^2_{\dot{H}^1_c} &= \|W_c\|^{\alpha^\star+2}_{L^{\alpha^\star+2}} = C_{\text{SE}}(c,\text{rad})^{-d}, \label{property W_c 1 radial}\\
E_c(W_c) &= \frac{1}{2}\|W_c\|^2_{\dot{H}^1_c} -\frac{1}{\alpha^\star+2} \|W_c\|^{\alpha^\star+2}_{L^{\alpha^\star+2}} = d^{-1}C_{\text{SE}}(c,\text{rad})^{-d}. \label{property W_c 2 radial} \\
C_{\text{SE}}(c,\text{rad})&=\|W_{c}\|^{-\frac{2}{d}}_{\dot{H}^1_c} = \|W_{c}\|^{-\frac{\alpha^\star+2}{d}}_{L^{\alpha^\star+2}} = [dE_{c}(W_{c})]^{-\frac{1}{d}}. \label{relation sharp sobolev embedding constant radial}
\end{align}
\section{Virial identities} \label{section virial identities}
\setcounter{equation}{0}
In this section, we derive virial identities and localized virial estimates associated to the $(\nlsct)$. Given a real valued function $\chi$, we define the virial potential by
\begin{align}
V_\chi(t):= \int \chi(x)|u(t,x)|^2 dx. \label{virial potential} 
\end{align}
By a direct computation, we have the following result.
\begin{lem} \label{lem derivative virial potential}
Let $d\geq 3$ and $c>-\lambda(d)$. If $u: I \times \R^d \rightarrow \C$ is a smooth-in-time and Schwartz-in-space solution to 
\[
i\partial_t u -P_c u = N(u),
\]
with $N(u)$ satisfying $\imemph{(N(u)\overline{u})}=0$, then we have for any $t\in I$,
\begin{align}
\frac{d}{dt} V_\chi (t)= 2 \int_{\R^d}\nabla \chi(x) \cdot  \imemph	{(\overline{u}(t,x) \nabla u(t,x))} dx, \label{first derivative viral potential}
\end{align}
and
\begin{equation}
\begin{aligned}
\frac{d^2}{dt^2} V_\chi(t) = &-\int \Delta^2 \chi(x) |u(t,x)|^2  dx  +  4 \sum_{j,k=1}^d \int \partial^2_{jk} \chi(x) \reemph{(\partial_k u(t,x) \partial_j \overline{u}(t,x))} dx  \\
&+4c\int \nabla \chi(x) \cdot \frac{x}{|x|^4} |u(t,x)|^2 dx+ 2\int \nabla \chi(x)\cdot \{N(u), u\}_p(t,x) dx, 
\end{aligned} \label{second derivative virial potential} 
\end{equation}
where $\{f, g\}_p :=\reemph{(f\nabla \overline{g} - g \nabla \overline{f})}$ is the momentum bracket.
\end{lem} 
We note that if $N(u) = - |u|^\alpha u$, then
\[
\{N(u), u\}_p =\frac{\alpha}{\alpha+2}\nabla(|u|^{\alpha+2}).
\]
Using this fact, we immediately have the following result.
\begin{coro}\label{coro derivative virial potential}
Let $d\geq 3$ and $c>-\lambda(d)$. If $u: I\times \R^d \rightarrow \C$ is a smooth-in-time and Schwartz-in-space solution to the $(\nlsce)$, then we have for any $t\in I$,
\begin{equation}
\begin{aligned}
\frac{d^2}{dt^2} V_\chi(t) =& -\int \Delta^2 \chi(x) |u(t,x)|^2  dx  +  4 \sum_{j,k=1}^d \int \partial^2_{jk} \chi(x) \reemph{(\partial_k u(t,x) \partial_j \overline{u}(t,x))} dx  \\
&+4c\int \nabla \chi(x) \cdot \frac{x}{|x|^{4}} |u(t,x)|^2 dx-\frac{2 \alpha}{\alpha+2} \int \Delta \chi(x) |u(t,x)|^{\alpha+2} dx. 
\end{aligned} \label{second derivative virial potential application} 
\end{equation}
\end{coro}
We now have the following standard virial identity for the $(\nlsct)$. 
\begin{lem} \label{lem global virial identity}
Let $d\geq 3$ and $c>-\lambda(d)$. Let $u_0 \in H^1$ be such that $|x|u_0 \in L^2$ and $u: I \times \R^d \rightarrow \C$ the corresponding solution to the $(\nlsce)$. Then, $|x| u \in C(I, L^2)$. Moreover, for any $t\in I$,
\begin{align}
\frac{d^2}{dt^2} \|x u(t)\|^2_{L^2} = 8\|u(t)\|^2_{\dot{H}^1_c} - \frac{4d\alpha}{\alpha+2}\|u(t)\|^{\alpha+2}_{L^{\alpha+2}}. \label{global virial identity}
\end{align}
\end{lem}
\begin{proof}
The first claim follows from the standard approximation argument, we omit the proof and refer the reader to \cite[Proposition 6.5.1]{Cazenave} for more details. It remains to show $(\ref{global virial identity})$. Applying Corollary $\ref{coro derivative virial potential}$ with $\chi(x)=|x|^2$, we have
\begin{align*}
\frac{d^2}{dt^2}V_{|x|^2}(t)=\frac{d^2}{dt^2} \|x u(t)\|^2_{L^2} &= 8 \int |\nabla u(t,x)|^2 + c  |x|^{-2}|u(t,x)|^2 dx - \frac{4d \alpha}{\alpha+2} \int |u(t,x)|^{\alpha+2}dx \\
&= 8 \|u(t)\|^2_{\dot{H}^1_c} - \frac{4d \alpha}{\alpha+2} \|u(t)\|^{\alpha+2}_{L^{\alpha+2}}.
\end{align*}
This gives $(\ref{global virial identity})$.
\end{proof}
In order to prove the blowup for the $(\nlsct)$ with radial data, we need localized virial estimates. To do so, we introduce a function $\theta: [0,\infty) \rightarrow [0,\infty)$ satisfying
\begin{align}
\theta(r) =\left\{\begin{array}{cl}
r^2 & \text{if } 0\leq r \leq 1, \\
\text{const.} &\text{if } r \geq 2, 
\end{array}
\right.
\quad \text{and} \quad \theta''(r) \leq 2 \quad \text{for }  r\geq 0.  \label{condition of varphi}
\end{align}
Note that the precise constant here is not important. For $R>1$, we define the radial function
\begin{align}
\varphi_R(x) = \varphi_R(r):=R^2 \theta(r/R), \quad r=|x|. \label{define rescaled varphi}
\end{align}
It is easy to see that
\begin{align}
2-\varphi''_R(r) \geq 0, \quad 2-\frac{\varphi'_R(r)}{r} \geq 0, \quad 2d- \Delta \varphi_R(x) \geq 0. \label{property rescaled varphi}
\end{align}
Here the last inequality follows from the fact $\Delta= \partial^2_r + \frac{d-1}{r} \partial_r$. 
\begin{lem} \label{lem localized virial identity}
Let $d\geq 3, c>-\lambda(d), R>1$ and $\varphi_R$ be as in $(\ref{define rescaled varphi})$. Let $u: I\times \R^d \rightarrow \C$ be a radial solution to the $(\nlsce)$. Then for any $t\in I$,
\begin{align}
\frac{d^2}{dt^2}V_{\varphi_R} (t)\leq 8\|u(t)\|^2_{\dot{H}^1_c} - \frac{4d\alpha}{\alpha+2} \|u(t)\|^{\alpha+2}_{L^{\alpha+2}} + O \Big( R^{-2} + R^{-\frac{(d-1)\alpha}{2}} \|u(t)\|^{\frac{\alpha}{2}}_{\dot{H}^1_c} \Big). \label{localized virial identity}
\end{align}
\end{lem}
\begin{proof}
We apply $(\ref{second derivative virial potential application})$ for $\chi(x) =\varphi_R(x)$ to get
\begin{align*}
\frac{d^2}{dt^2} V_{\varphi_R}(t) =& -\int \Delta^2 \varphi_R(x) |u(t,x)|^2  dx  +  4 \sum_{j,k=1}^d \int \partial^2_{jk} \varphi_R(x) \re{(\partial_k u(t,x) \partial_j \overline{u}(t,x))} dx  \\
&+4c\int \nabla \varphi_R(x) \cdot \frac{x}{|x|^{4}} |u(t,x)|^2 dx-\frac{2 \alpha}{\alpha+2} \int \Delta \varphi_R(x) |u(t,x)|^{\alpha+2} dx. 
\end{align*}
Since $\varphi_R(x)=|x|^2$ for $|x|\leq R$, we use $(\ref{global virial identity})$ to have
\begin{align}
\begin{aligned}
\frac{d^2}{dt^2}V_{\varphi_R} (t)&=8\|u(t)\|^2_{\dot{H}^1_c} - \frac{4d\alpha}{\alpha+2} \|u(t)\|^{\alpha+2}_{L^{\alpha+2}} - 8\|u(t)\|^2_{\dot{H}^1_c(|x|>R)} + \frac{4d\alpha}{\alpha+2} \|u(t)\|^{\alpha+2}_{L^{\alpha+2}(|x|>R)} \\
&\mathrel{\phantom{=}} -\int_{|x|>R} \Delta^2 \varphi_R |u(t)|^2  dx  +  4 \sum_{j,k=1}^d \int_{|x|>R} \partial^2_{jk} \varphi_R \re{(\partial_k u(t) \partial_j \overline{u}(t))} dx  \\
&\mathrel{\phantom{=}} +4c\int_{|x|>R} \nabla \varphi_R \cdot \frac{x}{|x|^{4}} |u(t)|^2 dx-\frac{2 \alpha}{\alpha+2} \int_{|x|>R} \Delta \varphi_R |u(t)|^{\alpha+2} dx. 
\end{aligned}
\label{localized virial estimate proof}
\end{align}
Since $|\Delta \varphi_R |\lesssim 1$ and $|\Delta^2 \varphi_R| \lesssim R^{-2}$, we have
\begin{align*}
\frac{d^2}{dt^2}V_{\varphi_R} (t)&= 8\|u(t)\|^2_{\dot{H}^1_c} - \frac{4d\alpha}{\alpha+2} \|u(t)\|^{\alpha+2}_{L^{\alpha+2}} + 4\sum_{j,k=1}^d \int_{|x|>R} \partial^2_{jk} \varphi_R \re{(\partial_k u(t) \partial_j \overline{u}(t))} dx \\
&\mathrel{\phantom{=8\|u(t)\|^2_{\dot{H}^1_c}}} + 4c \int_{|x|>R} \nabla \varphi_R \cdot \frac{x}{|x|^4} |u(t)|^2 dx - 8\|u(t)\|^2_{\dot{H}^1_c(|x|>R)} \\
&\mathrel{\phantom{=8\|u(t)\|^2_{\dot{H}^1_c}}} + O \Big( \int_{|x|> R} R^{-2} |u(t)|^2 + |u(t)|^{\alpha+2} dx \Big).
\end{align*}
Using $(\ref{property rescaled varphi})$ and the fact that
\[
\partial_j =\frac{x_j}{r} \partial_r, \quad \partial^2_{jk} = \Big(\frac{\delta_{jk}}{r}-\frac{x_j x_k}{r^3}\Big) \partial_r + \frac{x_j x_k}{r^2} \partial^2_r,
\]
we see that
\[
\sum_{j,k=1}^d \partial^2_{jk}\varphi_R \partial_k u \partial_j \overline{u} = \varphi''_R(r) |\partial_r u|^2 \leq 2 |\partial_r u|^2 = 2|\nabla u|^2,
\]
and 
\[
\nabla \varphi_R \cdot x = \varphi'_R \frac{x}{r} \cdot x = \varphi'_R r \leq 2r^2=2|x|^2. 
\]
Therefore
\[
4\sum_{j,k=1}^d \int_{|x|>R} \partial^2_{jk} \varphi_R \re{(\partial_k u \partial_j \overline{u})} dx + 4c \int_{|x|>R} \nabla \varphi_R \cdot x |x|^{-4} |u|^2 dx - 8\|u(t)\|^2_{\dot{H}^1_c(|x|>R)} \leq 0.
\]
The conservation of mass then implies
\begin{align*}
\frac{d^2}{dt^2} V_{\varphi_R}(t) &\leq 8\|u(t)\|^2_{\dot{H}^1_c} -\frac{4d\alpha}{\alpha+2} \|u(t)\|^{\alpha+2}_{L^{\alpha+2}} + O \Big( \int_{|x|> R} R^{-2} |u(t)|^2 + |u(t)|^{\alpha+2} dx \Big) \\
&\leq 8\|u(t)\|^2_{\dot{H}^1_c} -\frac{4d\alpha}{\alpha+2} \|u(t)\|^{\alpha+2}_{L^{\alpha+2}} + O\Big(R^{-2} + \|u(t)\|^{\alpha+2}_{L^{\alpha+2}(|x|>R)}\Big).
\end{align*}
It remains to bound $\|u(t)\|^{\alpha+2}_{L^{\alpha+2}(|x|>R)}$. To do this, we recall the following radial Sobolev embedding (\cite{Strauss, ChoOzawa}).
\begin{lem}[Radial Sobolev embedding \cite{Strauss, ChoOzawa}]
Let $d\geq 2$ and $\frac{1}{2}\leq s <1$. Then for any radial function $f$, 
\begin{align}
\sup_{x\ne 0} |x|^{\frac{d-2s}{2}}|f(x)| \leq C(d,s) \|f\|^{1-s}_{L^2} \|f\|^s_{\dot{H}^1}. \label{usual radial sobolev embedding}
\end{align}
Moreover, the above inequality also holds for $d\geq 3$ and $s=1$.
\end{lem}
Since $\dot{H}^1 \sim \dot{H}^1_c$, we have in particular
\begin{align}
\sup_{x\ne 0} |x|^{\frac{d-1}{2}}|f(x)| \lesssim \|f\|^{\frac{1}{2}}_{L^2} \|f\|^{\frac{1}{2}}_{\dot{H}^1_c}. \label{radial sobolev embedding inverse square}
\end{align}
Using $(\ref{radial sobolev embedding inverse square})$ and the conservation of mass, we estimate
\begin{align*}
\|u(t)\|^{\alpha+2}_{L^{\alpha+2}(|x|>R)} &\leq \Big(\sup_{|x|>R} |u(t)| \Big)^{\alpha} \|u(t)\|^2_{L^2} \\
&\lesssim R^{-\frac{(d-1)\alpha}{2}} \Big(\sup_{|x|>R} |x|^{\frac{d-1}{2}} |u(t)| \Big)^{\alpha} \|u(t)\|^2_{L^2} \\
&\lesssim R^{-\frac{(d-1)\alpha}{2}} \|u(t)\|^{\frac{\alpha}{2}}_{\dot{H}^1_c} \|u(t)\|^{\frac{\alpha}{2}+2}_{L^2} \lesssim R^{-\frac{(d-1)\alpha}{2}} \|u(t)\|^{\frac{\alpha}{2}}_{\dot{H}^1_c}.
\end{align*}
The proof is complete.
\end{proof}
The localized virial estimate given in Lemma $\ref{lem localized virial identity}$ is not enough to show blowup solutions in the mass-critical case, i.e. $\alpha=\alpha_\star$. In this case, we need a refined version of Lemma $\ref{lem localized virial identity}$. We follow the argument of \cite{OgawaTsutsumi1} (see also \cite{BoulengerHimmelsbachLenzmann}). 
\begin{lem} \label{lem localized virial identity mass-critical}
Let $d\geq 3, c>-\lambda(d), R>1$ and $\varphi_R$ be as in $(\ref{define rescaled varphi})$. Let $u: I\times \R^d \rightarrow \C$ be a radial solution to the mass-critical $(\nlsce)$, i.e. $\alpha=\alpha_\star$. Then for any $\ep>0$ and any $t\in I$,
\begin{align}
\frac{d^2}{dt^2}V_{\varphi_R} (t)\leq 16E_c(u_0) -4\int_{|x|>R} \Big(\chi_{1,R} - \frac{\ep}{d+2} \chi_{2,R}^{\frac{d}{2}} \Big) |\nabla u(t)|^2dx + O \Big( R^{-2} +\ep R^{-2} + \ep^{-\frac{2}{d-2}}R^{-2} \Big), \label{localized virial identity mass-critical}
\end{align}
where
\begin{align}
\chi_{1,R} = 2-\varphi''_R, \quad \chi_{2,R}= 2d-\Delta \varphi_R. \label{define chi}
\end{align}
\end{lem}
\begin{proof}
Using $(\ref{localized virial estimate proof})$ with $\alpha=\alpha_\star =\frac{4}{d}$ and $\sum_{j,k} \partial^2_{jk} \varphi_R \partial_k u \partial_j \overline{u} = \varphi''_R |\partial_r u|^2$ and rewriting $\varphi''_R = 2 - (2-\varphi''_R)$ and $\Delta \varphi_R = 2d - (2d-\Delta \varphi_R)$, we have
\begin{align*}
\frac{d^2}{dt^2}V_{\varphi_R}(t)&= 16E_c(u(t)) -\int_{|x|>R} \Delta^2\varphi_R |u(t)|^2 dx - 4 \int_{|x|>R} (2-\varphi''_R) |\partial_r u(t)|^2 dx  \\
&\mathrel{\phantom{=16 E_c(u(t)) -\int_{|x|>R} \Delta^2\varphi_R |u(t)|^2 dx }} +\frac{4}{d+2} \int_{|x|>R} (2d-\Delta \varphi_R) |u(t)|^{\frac{4}{d}+2} dx \\
&\mathrel{\phantom{=}} + 8\int_{|x|>R} |\partial_r u(t)|^2 dx +4c\int_{|x|>R} \nabla \varphi_R \cdot x|x|^{-4} |u(t)|^2 dx -8\|u(t)\|^2_{\dot{H}^1_c(|x|>R)} \\
&\leq 16 E_c(u_0) + O(R^{-2}) - 4 \int_{|x|>R}\chi_{1,R} |\nabla u(t)|^2 dx   +\frac{4}{d+2} \int_{|x|>R} \chi_{2,R} |u(t)|^{\frac{4}{d}+2} dx.
\end{align*}
We now bound the last term. Using the radial Sobolev embedding $(\ref{usual radial sobolev embedding})$ with $s=1$ and the conservation of mass, we estimate
\begin{align*}
\int_{|x|>R} \chi_{2,R} |u(t)|^{\frac{4}{d}+2} dx & = \int_{|x|>R} |\chi_{2,R}^{\frac{d}{4}} u(t)|^{\frac{4}{d}} |u(t)|^2 dx \\
&\leq \Big(\sup_{|x|>R} |\chi^{\frac{d}{4}}_{2,R}(x) u(t,x)|\Big)^{\frac{4}{d}} \|u(t)\|^2_{L^2} \\
& \lesssim R^{-\frac{2(d-2)}{d}} \Big\| \nabla\Big(\chi_{2,R}^{\frac{d}{4}} u(t) \Big) \Big\|^{\frac{4}{d}}_{L^2} \|u(t)\|^2_{L^2} \\
&\lesssim R^{-\frac{2(d-2)}{d}} \Big\| \nabla\Big(\chi_{2,R}^{\frac{d}{4}} u(t) \Big) \Big\|^{\frac{4}{d}}_{L^2}.
\end{align*}
We next use the Young inequality $ab \lesssim \ep a^p + \ep^{-\frac{q}{p}} b^q$ with $\frac{1}{p}+\frac{1}{q}=1$ and $\ep>0$ an arbitrary real number to have
\[
R^{-\frac{2(d-2)}{d}} \Big\| \nabla\Big(\chi_{2,R}^{\frac{d}{4}} u(t) \Big) \Big\|^{\frac{4}{d}}_{L^2} \lesssim \ep \Big\| \nabla\Big(\chi_{2,R}^{\frac{d}{4}} u(t) \Big) \Big\|^2_{L^2} + O\Big( \ep^{-\frac{2}{d-2}} R^{-2}\Big).
\] 
Here we apply the Young inequality with $p=\frac{d}{2}$ and $q=\frac{d}{d-2}$. It is not hard to check $|\nabla(\chi_{2,R}^{d/4} )| \lesssim R^{-1}$ for $|x|>R$. Thus the conservation of mass implies
\[
\Big\| \nabla\Big(\chi_{2,R}^{\frac{d}{4}} u(t) \Big) \Big\|^2_{L^2} \lesssim R^{-2} + \Big\|\chi_{2,R}^{\frac{d}{4}} \nabla u(t) \Big\|^2_{L^2}.
\]
Combining the above estimates, we prove $(\ref{localized virial identity mass-critical})$.
\end{proof}
\section{Global existence} \label{section global existence}
\setcounter{equation}{0}
In this section, we give the proofs of global existence given Theorem $\ref{theorem global blowup NLS mass-critical inverse square}$ and Theorem $\ref{theorem global blowup NLS intercritical inverse square}$. 
\subsection{Mass-critical case} \label{subsection global mass-critical}
Thanks to the local well-posedness given in Theorem $\ref{theorem energy method Okazawa-Suzuki-Yokota}$, it suffices to bound $\|u(t)\|_{H^1_c}$ for all $t$ in the existence time. Applying $(\ref{sharp gagliardo nirenberg constant})$ with $\alpha=\alpha_\star$, we see that 
\[
C_{\text{GN}}(c) = \frac{\alpha_\star+2}{2\|Q_{\cbar}\|^{\alpha_\star}_{L^2}}.
\] 
By the definition of energy, we have
\[
\|u(t)\|^2_{\dot{H}^1_c} = 2 E_c(u(t)) +\frac{2}{\alpha_\star+2} \|u(t)\|^{\alpha_\star+2}_{L^{\alpha_\star+2}}.
\]
The sharp Gagliardo-Nirenberg inequality and the conservations of mass and energy imply
\begin{align*}
\|u(t)\|^2_{\dot{H}^1_c} &\leq 2 E_c(u(t)) + \frac{2}{\alpha_\star+2} C_{\text{GN}}(c) \|u(t)\|^{\alpha_\star}_{L^2}\|u(t)\|^2_{\dot{H}^1_c} \\
& = 2 E_c(u_0) + \frac{2}{\alpha_\star+2} C_{\text{GN}}(c) \|u_0|^{\alpha_\star}_{L^2}\|u(t)\|^2_{\dot{H}^1_c} \\
& = 2 E_c(u_0) + \Big(\frac{\|u_0\|_{L^2}}{\|Q_{\cbar}\|_{L^2}}\Big)^{\alpha_\star} \|u(t)\|^2_{\dot{H}^1_c}.
\end{align*}
Thus, 
\[
\Big[1- \Big(\frac{\|u_0\|_{L^2}}{\|Q_{\cbar}\|_{L^2}}\Big)^{\alpha_\star} \Big]\|u(t)\|^2_{\dot{H}^1_c} \leq 2 E_c(u_0).
\]
Since $\|u_0\|_{L^2}<\|Q_{\cbar}\|_{L^2}$, the above estimate shows the boundedness of $\|u(t)\|_{\dot{H}^1_c}$. Hence $\|u(t)\|_{H^1_c}$ is bounded by the conservation of mass. This proves the global existence of Theorem $\ref{theorem global blowup NLS mass-critical inverse square}$. 
\begin{rem} \label{rem proof global NLS mass-critical inverse square}
Let us show Item 3 of Remark $\ref{rem global blowup NLS mass-critical inverse square}$. Let $-\lambda(d)<c<0$ and $M_c > \|Q_c\|_{L^2}$. Let $\lambda = M_c/\|Q_c\|_{L^2}>1$. Set $u_0(x)= \lambda Q_c(x)$. We have $\|u_0\|_{L^2} = M_c$ and 
\begin{align*}
E_c(u_0) = E_c(\lambda Q_c) &= \frac{\lambda^2}{2}\|Q_c\|^2_{\dot{H}^1_c} - \frac{\lambda^{\alpha_\star+2}}{\alpha_\star+2} \|Q_c\|^{\alpha_\star+2}_{L^{\alpha_\star+2}} \\
&=\lambda^{\alpha_\star+2} E_c(Q_c) -\frac{\lambda^{\alpha_\star+2} - \lambda^2}{2} \|Q_c\|^2_{\dot{H}^1_c}.
\end{align*}
Since $E_c(Q_c)=0$ and $\lambda>1$, we see that $E_c(u_0)<0$. On the other hand, it is obvious that $u_0$ is radial. Thus by Item 2 of Theorem $\ref{theorem global blowup NLS mass-critical inverse square}$, we see that the corresponding solution with initial data $u_0$ blows up in finite time. \newline
\indent We next show for $c>0$ that if $u_0$ is radial and satisfies $\|u_0\|_{L^2}<\|Q_{c,\text{rad}}\|_{L^2}$, then the corresponding solution exists globally. It follows similarly as the beginning of Subsection $\ref{subsection global mass-critical}$ by using the sharp radial Gagliardo-Nirenberg inequality
\[
\|f\|_{L^{\alpha_\star+2}}^{\alpha_\star+2} \leq C_{\text{GN}}(c,\text{rad}) \|f\|^{\alpha_\star}_{L^2} \|f\|^2_{\dot{H}^1_c}, \quad f \text{ radial}.
\]
Note also that by Item 4 of Remark $\ref{rem ground state equation}$, we have
\[
C_{\text{GN}}(c,\text{rad}) = \frac{\alpha_\star+2}{2\|Q_{c,\text{rad}}\|^{\alpha_\star}_{L^2}}.
\]
To complete the proof of Item 3, we show that for any $M_c>\|Q_{c,\text{rad}}\|_{L^2}$, there exists $u_0\in H^1$ radial satisfying $\|u_0\|_{L^2} = M_c$ and the corresponding solution blows up in finite time. We proceed as above. Let $\lambda = M_c/\|Q_{c,\text{rad}}\|_{L^2}>1$ and set $u_0(x) = \lambda Q_{c,\text{rad}}(x)$. We see that $\|u_0\|_{L^2} = M_c$ and
\begin{align*}
E_c(u_0) = E_c(\lambda Q_{c, \text{rad}}) &= = \frac{\lambda^2}{2}\|Q_{c,\text{rad}}\|^2_{\dot{H}^1_c} - \frac{\lambda^{\alpha_\star+2}}{\alpha_\star+2} \|Q_{c,\text{rad}}\|^{\alpha_\star+2}_{L^{\alpha_\star+2}} \\
&=\lambda^{\alpha_\star+2} E_c(Q_{c,\text{rad}}) -\frac{\lambda^{\alpha_\star+2} - \lambda^2}{2} \|Q_{c,\text{rad}}\|^2_{\dot{H}^1_c}.
\end{align*} 
Since $Q_{c,\text{rad}}$ is a solution to the $(\ref{ground state equation radial})$, we see that $E_c(Q_{c,\text{rad}})=0$. This shows that $E_c(u_0)<0$. Thus the corresponding solution blows up in finite time. 
\end{rem}
\begin{rem} \label{rem existence blowup solution mass-critical radial}
Let us show Item 5 of Remark $\ref{rem global blowup NLS mass-critical inverse square}$, that is to show when $c>0$ there exists a radial blowup solution to the mass-critical $(\nlsct)$ with $\|u_0\|_{L^2}=\|Q_{c,\text{rad}}\|_{L^2}$. Since $Q_{c,\text{rad}}$ is a solution to the elliptic equation
\[
-P_c Q_{c,\text{rad}} -Q_{c,\text{rad}} + Q_{c,\text{rad}}^{\alpha_\star+1}=0,
\]
it is easy to see that $u(t) = e^{it}Q_{c,\text{rad}}$ is a solution to the mass-critical $(\nlsct)$. Then a direct computation shows that for any $0<T<+\infty$, the function 
\[
u_T(t,x) =  \frac{1}{|t-T|^{d/2}} e^{-i\frac{|x|^2}{4(t-T)} + \frac{i}{t-T}} Q_{c,\text{rad}}\Big(\frac{x}{t-T}\Big)
\]
 is also a solution to the mass-critical $(\nlsct)$ which blows up at $T$ and $\|u_T(0)\|_{L^2} = \|Q_{c,\text{rad}}\|_{L^2}$. 
\end{rem}
\subsection{Intercritical case}
Again thanks to the local well-posedness of the $(\nlsct)$ given in Theorem $\ref{theorem energy method Okazawa-Suzuki-Yokota}$. It suffices to show that $\|u(t)\|_{H^1_c}$ is bounded as long as $t$ belongs to the existence time. Let $u_0 \in H^1$ be such that $(\ref{condition below ground state NLS intercritical inverse square})$ and $(\ref{condition global existence NLS intercritical inverse square})$ hold. By the definition of energy and multiplying both sides of $E_c(u(t))$ by $M(u(t))^\sigma$, the sharp Gagliardo-Nirenberg inequality $(\ref{sharp gagliardo nirenberg inequality})$ implies
\begin{align}
E_c(u(t)) M(u(t))^\sigma &=\frac{1}{2} \Big( \|u(t)\|_{\dot{H}^1_c} \|u(t)\|^\sigma_{L^2}\Big)^2 - \frac{1}{\alpha+2} \|u(t)\|^{\alpha+2}_{L^{\alpha+2}} \|u(t)\|^{2\sigma}_{L^2} \nonumber \\
& \geq \frac{1}{2} \Big( \|u(t)\|_{\dot{H}^1_c} \|u(t)\|^\sigma_{L^2}\Big)^2 - \frac{C_{\text{GN}}(c)}{\alpha+2} \|u(t)\|^{\frac{4-(d-2)\alpha}{2} + 2\sigma}_{L^2} \|u(t)\|^{\frac{d\alpha}{2}}_{\dot{H}^1_c} \nonumber \\
& = f(\|u(t)\|_{\dot{H}^1_c} \|u(t)\|^\sigma_{L^2}), \label{estimate f}
\end{align}
where 
\begin{align}
f(x)=\frac{1}{2} x^2 -\frac{C_{\text{GN}}(c)}{\alpha+2}x^{\frac{d\alpha}{2}}. \label{define f}
\end{align}
Using $(\ref{relation kinetic})$ and $(\ref{relation energy kinetic})$, we see that
\begin{align}
f(K(c)) = \frac{d\alpha-4}{2d\alpha} K(c)^2 = H(c). \label{value f}
\end{align}
We have from $(\ref{estimate f})$, the conservations of mass and energy and the assumption $(\ref{condition below ground state NLS intercritical inverse square})$ that
\begin{align}
f(\|u(t)\|_{\dot{H}^1_c} \|u(t)\|^\sigma_{L^2})\leq E_c(u_0) M(u_0)^\sigma <H(c). \label{estimate f global intercritical}
\end{align}
Using this together with $(\ref{condition global existence NLS intercritical inverse square})$, $(\ref{value f})$ and $(\ref{estimate f global intercritical})$, the continuity argument shows
\[
\|u(t)\|_{\dot{H}^1_c} \|u(t)\|^\sigma_{L^2} <K(c),
\]
for any $t$ as long as the solution exists. The conservation of mass then implies the boundedness of $\|u(t)\|_{H^1_c}$. \newline
\indent The global existence of Theorem $\ref{theorem global blowup NLS intercritical inverse square radial}$ is proved similarly as above using Item 4 of Remark $\ref{rem ground state equation}$.
\section{Blowup} \label{section blowup}
\setcounter{equation}{0}
This section is devoted to the proofs of blowup solutions given in Theorem $\ref{theorem global blowup NLS mass-critical inverse square}$, Theorem $\ref{theorem global blowup NLS intercritical inverse square}$ and Theorem $\ref{theorem blowup NLS energy-critical inverse square}$. 
\subsection{Mass-critical case} 
Let us consider the case $E_c(u_0)<0$ and $xu_0 \in L^2$. By the standard virial identity $(\ref{global virial identity})$, 
\[
\frac{d^2}{dt^2} \|x u(t)\|^2_{L^2} = 8\|u(t)\|^2_{\dot{H}^1_c} -\frac{4d\alpha_\star}{\alpha_\star+2} \|u(t)\|^{\alpha_\star+2}_{L^{\alpha_\star+2}} = 16 E_c(u_0)<0.
\]
By the classical argument of Glassey \cite{Glassey}, it follows that the solution $u$ blows up in finite time. \newline
\indent We next consider the case $E_c(u_0)<0$ and $u_0$ is radial. Applying the localized virial estimate $(\ref{localized virial identity mass-critical})$, we have
\begin{align*}
\frac{d^2}{dt^2}V_{\varphi_R} (t)\leq 16E_c(u_0) -4\int_{|x|>R} \Big(\chi_{1,R} - \frac{\ep}{d+2} \chi_{2,R}^{\frac{d}{2}} \Big) |\nabla u(t)|^2dx + O \Big( R^{-2} + \ep^{-\frac{2}{d-2}}R^{-2} \Big), 
\end{align*}
where $\chi_{1,R} = 2-\varphi''_R$ and $\chi_{2,R}= 2d-\Delta \varphi_R$. We seek for a radial function $\varphi_R$ defined by $(\ref{define rescaled varphi})$ so that
\begin{align}
\chi_{1,R} - \frac{\ep}{d+2} \chi_{2,R}^{\frac{d}{2}} \geq 0, \quad \forall r>R, \label{positive condition}
\end{align}
for a sufficiently small $\ep>0$. If $(\ref{positive condition})$ is satisfied, then by choosing $R>1$ sufficiently large depending on $\ep$, we see that
\[
\frac{d^2}{dt^2}V_{\varphi_R} (t)\leq 8E_c(u_0)<0,
\]
for any $t$ in the existence time. This shows that the solution $u$ must blow up in finite time. It remains to show $(\ref{positive condition})$. To do so, we follow the argument of \cite{OgawaTsutsumi1}. Let us define a function
\[
\vartheta(r):= \left\{
\begin{array}{c l}
2r & \text{if } 0\leq r\leq 1, \\
2[r-(r-1)^3] &\text{if } 1<r\leq 1+1/\sqrt{3}, \\
\vartheta' <0 &\text{if } 1+ 1/\sqrt{3} <r < 2, \\
0 &\text{if } r\geq 2,
\end{array}
\right.
\]
and 
\begin{align*}
\theta(r):= \int_{0}^{r}\vartheta(s)ds.
\end{align*}
It is easy to see that $\theta$ satisfies $(\ref{condition of varphi})$. Define $\varphi_R$ as in $(\ref{define rescaled varphi})$. We will show that $(\ref{positive condition})$ holds true for this choice of $\varphi_R$. Indeed, by definition,
\[
\varphi'_R(r) = R\theta'(r/R) = R \vartheta(r/R), \quad \varphi''_R(r) = \theta''(r/R) = \vartheta'(r/R), \quad \Delta \varphi_R(x) = \varphi''_R(r) +\frac{d-1}{r}\varphi'_R(r).
\]
\indent When $r> (1+1/\sqrt{3})R$, we see that $\vartheta'(r/R) \leq 0$, so $\chi_{1,R}(r) = 2-\varphi''_R(r) \geq 2$. We also have $\chi_{2,R}(r) \leq C$ for some constant $C>0$. Thus by choosing $\ep>0$ small enough, we have $(\ref{positive condition})$. \newline
\indent When $R<r \leq (1+1/\sqrt{3})R$, we have
\[
\chi_{1,R}(r) = 6 \Big(\frac{r}{R}-1\Big)^2, \quad \chi_{2,R}(r) = 6\Big(\frac{r}{R}-1\Big)^2 \Big[1 + \frac{(d-1)(r/R-1)}{3r/R} \Big] <6\Big(\frac{r}{R}-1\Big)^2 \Big(1+ \frac{d-1}{3\sqrt{3}}\Big).
\]
Since $0<r/R-1<1/\sqrt{3}$, we can choose $\ep>0$ small enough, for instance,
\[
\ep < (d+2)\Big(1+ \frac{d-1}{3\sqrt{3}}\Big)^{-d/2}
\]
to get $(\ref{positive condition})$. The proof is complete. \defendproof
\begin{rem} \label{rem necessary condition blowup mass-critical inverse square}
We now show Item 4 of Remark $\ref{rem global blowup NLS mass-critical inverse square}$ that is to show the condition $E_c(u_0)<0$ is a sufficient condition but it is not necessary. Let $E_c>0$. We find data $u_0 \in H^1$ so that $E_c(u_0)=E_c$ and the corresponding solution $u$ blows up in finite time. We follow the standard argument (see e.g. \cite[Remark 6.5.8]{Cazenave}). Using the standard virial identity with $\alpha=\alpha_\star$, we have
\[
\frac{d^2}{dt^2}\|x u(t)\|^2_{L^2} = 16 E_c(u_0),
\]
hence 
\[
\|x u(t)\|^2_{L^2} = 8 t^2 E_c(u_0) + 4t \Big(\im{\int \overline{u}_0 x \cdot \nabla u_0} dx\Big) + \|x u_0\|_{L^2}^2=:f(t).
\]  
Note that if $f(t)$ takes negative values, then the solution $u$ must blow up in finite time. In order to make $f(t)$ takes negative values, we need
\begin{align}
\Big(\im{\int \overline{u}_0 x \cdot \nabla u_0} dx \Big)^2 > 2E_c(u_0) \|x u_0\|^2_{L^2}. \label{sufficient and necessary condition}
\end{align}
Now fix $\theta \in C^\infty_0(\R^d)$ a real-valued function and set $\psi(x)=e^{-i|x|^2} \theta(x)$. We see that $\psi \in C^\infty_0(\R^d)$ and 
\[
\im{\int \overline{\psi} x\cdot \nabla \psi dx} = -2 \int  |x|^2 \theta^2(x) dx <0. 
\]
We now set
\begin{align*}
\begin{aligned}
A &= \frac{1}{2} \|\psi\|^2_{\dot{H}^1_c},& B &= \frac{1}{\alpha_\star +2} \|\psi\|^{\alpha_\star+2}_{L^{\alpha_\star +2}}, \\
C &=\|x\psi\|^2_{L^2}, &  D &= -\im{\int \overline{\psi} x\cdot \nabla \psi dx}.
\end{aligned}
\end{align*}
Let $\lambda, \mu>0$ be chosen later and set $u_0(x) =\lambda \psi(\mu x)$. We will choose $\lambda, \mu>0$ so that $E_c(u_0) = E_c$ and $(\ref{sufficient and necessary condition})$ holds true. A direct computation shows
\[
E_c(u_0) = \lambda^2 \mu^2 \mu^{-d} \frac{1}{2}\|\psi\|^2_{\dot{H}^1_c} - \lambda^{\alpha_\star+2} \mu^{-d} \frac{1}{\alpha_\star+2} \|\psi\|^{\alpha_\star+2}_{L^{\alpha_\star+2}} = \lambda^2 \mu^{2-d} \Big(A-\frac{\lambda^{\alpha_\star}}{\mu^2} B\Big),
\]
and
\[
\im{\int \overline{u}_0 x \cdot\nabla u_0 dx} = \lambda^2 \mu^{-d} \im{\int \overline{\psi} x \cdot \nabla \psi dx} = -\lambda^2 \mu^{-d} D,
\]
and
\[
\|xu_0\|^2_{L^2} = \lambda^2 \mu^{-d-2} \|x\psi\|^2_{L^2} = \lambda^2 \mu^{-d-2} C.
\]
Thus, the conditions $E_c(u_0)=E_c$ and $(\ref{sufficient and necessary condition})$ yield
\begin{align}
\lambda^2 \mu^{2-d} &\Big(A-\frac{\lambda^{\alpha_\star}}{\mu^2}B\Big) = E_c, \label{condition 1}\\
\frac{D^2}{C}&> 2\Big(A-\frac{\lambda^{\alpha_\star}}{\mu^2}B\Big). \label{condition 2}
\end{align}
Fix $0<\ep < \min\Big\{A, \frac{D^2}{2C} \Big\}$ and choose
\[
\frac{\lambda^{\alpha_\star}}{\mu^2} B =A-\ep.
\]
It is obvious that $(\ref{condition 2})$ is satisfied. Condition $(\ref{condition 1})$ implies
\[
\ep\lambda^2 \mu^{2-d} = E_c \quad \text{or} \quad \ep \Big(\frac{B}{A-\ep}\Big)^{\frac{2-d}{2}} \lambda^{2+\frac{(2-d)\alpha_\star}{2}} = E_c.
\]
This holds true by choosing a suitable value of $\lambda$. 
\end{rem}
\subsection{Intercritical case}
We first consider the case $E_c(u_0) \geq 0$. We first show $(\ref{property blowup solution NLS intercritical inverse square})$. We have from $(\ref{estimate f})$ that
\[
f(\|u(t)\|_{\dot{H}^1_c} \|u(t)\|^\sigma_{L^2}) \leq E_c(u(t)) M(u(t))^\sigma,
\]
where $f$ is defined as in $(\ref{define f})$. Note that $f(K(c)) = H(c)$. By our assumption $(\ref{condition below ground state NLS intercritical inverse square})$, we have
\[
f(\|u(t)\|_{\dot{H}^1_c} \|u(t)\|^\sigma_{L^2}) < H(c).
\]
Using $(\ref{condition blowup NLS intercritical inverse square})$ and the continuity argument, we get
\[
\|u(t)\|_{\dot{H}^1_c} \|u(t)\|^\sigma_{L^2} > K(c),
\]
for any $t$ in the existence time. This proves $(\ref{property blowup solution NLS intercritical inverse square})$. \newline
\indent We next pick $\delta>0$ small enough so that 
\begin{align}
E_c(u_0) M(u_0)^\sigma \leq (1-\delta) H(c). \label{refined estimate condition}
\end{align} 
This implies 
\begin{align}
f(\|u(t)\|_{\dot{H}^1_c}\|u(t)\|^\sigma_{L^2}) \leq (1-\delta) H(c). \label{refined estimate f}
\end{align}
Using $(\ref{define f}), (\ref{relation kinetic})$ and $(\ref{relation energy kinetic})$, we have from $(\ref{refined estimate f})$ that
\[
\frac{d\alpha}{d\alpha-4} \Big(\frac{\|u(t)\|_{\dot{H}^1_c} \|u(t)\|^\sigma_{L^2}}{K(c)} \Big)^2 - \frac{4}{d\alpha-4} \Big( \frac{\|u(t)\|_{\dot{H}^1_c} \|u(t)\|^\sigma_{L^2}}{K(c)} \Big)^{\frac{d\alpha}{2}} \leq 1-\delta.
\]
The continuity argument shows that there exists $\delta'>0$ depending on $\delta$ so that
\begin{align}
\frac{\|u(t)\|_{\dot{H}^1_c} \|u(t)\|^\sigma_{L^2}}{K(c)} \geq 1+\delta' \quad \text{or}\quad \|u(t)\|_{\dot{H}^1_c} \|u(t)\|^\sigma_{L^2} \geq (1+\delta') K(c). \label{refined property blowup intercritical}
\end{align}
\indent We also have for $\ep>0$ small enough,
\begin{align}
8\|u(t)\|^2_{\dot{H}^1_c} -\frac{4d\alpha}{\alpha+2} \|u(t)\|^{\alpha+2}_{L^{\alpha+2}} + \ep \|u(t)\|^2_{\dot{H}^1_c} \leq -c<0, \label{refined virial estimate intercritical}
\end{align}
for any $t$ in the existence time. Indeed, multiplying the left hand side of $(\ref{refined virial estimate intercritical})$ with a conserved quantity $M(u(t))^\sigma$, we get
\[
\text{LHS}(\ref{refined virial estimate intercritical}) \times M(u(t))^\sigma = 4d\alpha E_c(u(t)) M(u(t))^\sigma + (8+\ep-2d\alpha) \|u(t)\|^2_{\dot{H}^1_c} M(u(t))^\sigma.
\] 
The conservations of mass and energy, $(\ref{refined estimate condition}), (\ref{refined property blowup intercritical})$ and $(\ref{relation energy kinetic})$ then yield
\begin{align*}
\text{LHS}(\ref{refined virial estimate intercritical}) \times M(u_0)^\sigma &\leq 4d\alpha (1-\delta) H(c) + (8+\ep -2d\alpha) (1+\delta')^2 K(c)^2 \\
&= 2(d\alpha-4) (1-\delta) K(c)^2 + (8+\ep-2d\alpha) (1+\delta')^2 K(c)^2 \\
&=  K(c)^2 \Big[ 2(d\alpha-4)(1-\delta - (1+\delta')^2) + \ep (1+\delta')^2 \Big].
\end{align*}
By taking $\ep>0$ small enough, we prove $(\ref{refined virial estimate intercritical})$. \newline
\indent Let us consider the case $xu_0 \in L^2$ satisfying $(\ref{condition below ground state NLS intercritical inverse square})$ and $(\ref{condition blowup NLS intercritical inverse square})$. By the standard virial identity $(\ref{global virial identity})$ and $(\ref{refined virial estimate intercritical})$,
\[
\frac{d^2}{dt^2} \|x u(t)\|^2_{L^2} = 8 \|u(t)\|^2_{\dot{H}^1_c} -\frac{4d\alpha}{\alpha+2} \|u(t)\|^{\alpha+2}_{L^{\alpha+2}} \leq -c <0. 
\]
This shows that the solution blows up in finite time. \newline
\indent We now consider the case $u_0$ is radial, and satisfies $(\ref{condition below ground state NLS intercritical inverse square})$ and $(\ref{condition blowup NLS intercritical inverse square})$. Using the localized virial estimate $(\ref{localized virial identity})$, we have
\[
\frac{d^2}{dt^2}V_{\varphi_R} (t)\leq 8\|u(t)\|^2_{\dot{H}^1_c} - \frac{4d\alpha}{\alpha+2} \|u(t)\|^{\alpha+2}_{L^{\alpha+2}} + O \Big( R^{-2} + R^{-\frac{(d-1)\alpha}{2}} \|u(t)\|^{\frac{\alpha}{2}}_{\dot{H}^1_c} \Big).
\]
We next use the Young inequality to bound
\[
R^{-\frac{(d-1)\alpha}{2}} \|u(t)\|^{\frac{\alpha}{2}}_{\dot{H}^1_c} \lesssim \ep \|u(t)\|^2_{\dot{H}^1_c} + \ep^{-\frac{\alpha}{4-\alpha}} R^{-\frac{2(d-1)\alpha}{4-\alpha}},
\]
for $\ep>0$ an arbitrary real number. We thus get
\[
\frac{d^2}{dt^2}V_{\varphi_R} (t)\leq 8\|u(t)\|^2_{\dot{H}^1_c} - \frac{4d\alpha}{\alpha+2} \|u(t)\|^{\alpha+2}_{L^{\alpha+2}} + \ep \|u(t)\|^2_{\dot{H}^1_c} + O\Big(R^{-2} +\ep^{-\frac{\alpha}{4-\alpha}} R^{-\frac{2(d-1)\alpha}{4-\alpha}}\Big).
\]
By taking $\ep>0$ small enough and $R>1$ large enough depending on $\ep$, we obtain from $(\ref{refined virial estimate intercritical})$ that
\[
\frac{d^2}{dt^2}V_{\varphi_R} (t)\leq -c/2<0.
\]
This shows that the solution must blow up in finite time. \newline
\indent The case $E_c(u_0)<0$ is similar and easier. We simply bound
\[
\frac{d^2}{dt^2}V_{\varphi_R}(t) \leq 4d\alpha E_c(u_0) - 2(d\alpha-4) \|u(t)\|^2_{\dot{H}^1_c} + \ep \|u(t)\|^2_{\dot{H}^1_c} + O \left(R^{-2} + \ep^{-\frac{\alpha}{4-\alpha}} R^{-\frac{2(d-1)\alpha}{4-\alpha}} \right).
\]
Since $d\alpha >4$, we can choose $\ep>0$ small enough so that 
\[
\frac{d^2}{dt^2} V_{\varphi_R}(t) \leq 4d\alpha E_c(u_0) + O \left(R^{-2} + \ep^{-\frac{\alpha}{4-\alpha}} R^{-\frac{2(d-1)\alpha}{4-\alpha}} \right).
\]
Taking $R>1$ large enough depending on $\ep$, we get
\[
\frac{d^2}{dt^2} V_{\varphi_R}(t) \leq 2d\alpha E_c(u_0) <0.
\]
This again implies that the solution must blow up in finite time. \newline
\indent The blowup of Theorem $\ref{theorem global blowup NLS intercritical inverse square radial}$ follows by the same argument as above and Item 4 of Remark $\ref{rem ground state equation}$.
\subsection{Energy-critical case}
We again only consider the case $E_c(u_0) \geq 0$, the case $E_c(u_0)<0$ is similar to the intercritical case. By definition of the energy and the sharp Sobolev embedding inequality $(\ref{sharp sobolev embedding})$,
\begin{align*}
E_c(u(t)) &=\frac{1}{2}\|u(t)\|^2_{\dot{H}^1_c} -\frac{1}{\alpha^\star+2} \|u(t)\|^{\alpha^\star+2}_{L^{\alpha^\star+2}} \\
&\leq \frac{1}{2}\|u(t)\|^2_{\dot{H}^1_c} -\frac{[C_{\text{SE}}(c)]^{\alpha^\star+2}}{\alpha^\star+2} \|u(t)\|^{\alpha^\star+2}_{\dot{H}^1_c} =: g(\|u(t)\|_{\dot{H}^1_c}),
\end{align*}
where 
\begin{align}
g(y)=\frac{1}{2}y^2 -\frac{[C_{\text{SE}}(c)]^{\alpha^\star+2}}{\alpha^\star+2} y^{\alpha^\star+2}. \label{define g}
\end{align}
We have from $(\ref{relation sharp sobolev embedding constant})$ that
\[
g(\|W_{\cbar}\|_{\dot{H}^1_{\cbar}}) = E_{\cbar}(W_{\cbar}).
\]
By the conservation of energy and the assumption $E_c(u_0)<E_{\cbar}(W_{\cbar})$, 
\[
g(\|u(t)\|_{\dot{H}^1_c}) \leq E_c(u(t)) = E_c(u_0)<E_{\cbar}(W_{\cbar}).
\]
We thus have from the assumption $\|u_0\|_{\dot{H}^1_c}> \|W_{\cbar}\|_{\dot{H}^1_{\cbar}}$ and the continuity argument that
\begin{align}
\|u(t)\|_{\dot{H}^1_c}> \|W_{\cbar}\|_{\dot{H}^1_{\cbar}}, \label{blowup energy-critical estimate}
\end{align}
for any $t$ as long as the solution exists. We next improve $(\ref{blowup energy-critical estimate})$ as follows. Pick $\delta>0$ small enough so that 
\begin{align}
E_c(u_0) \leq (1-\delta) E_{\cbar}(W_{\cbar}). \label{refined blowup energy-critical estimate condition}
\end{align} 
This implies 
\begin{align}
g(\|u(t)\|_{\dot{H}^1_c}) \leq (1-\delta) E_{\cbar}(W_{\cbar}). \label{refined blowup energy-critical estimate g}
\end{align}
Using $(\ref{define g})$ and $(\ref{relation sharp sobolev embedding constant})$, we have from $(\ref{refined blowup energy-critical estimate g})$ that
\[
\frac{d}{2} \Big(\frac{\|u(t)\|_{\dot{H}^1_c}}{\|W_{\cbar}\|_{\dot{H}^1_{\cbar}}} \Big)^2 - \frac{d-2}{2} \Big( \frac{\|u(t)\|_{\dot{H}^1_c}}{\|W_{\cbar}\|_{\dot{H}^1_{\cbar}}} \Big)^{\alpha^\star+2} \leq 1-\delta.
\]
The continuity argument shows that there exists $\delta'>0$ depending on $\delta$ so that
\begin{align}
\frac{\|u(t)\|_{\dot{H}^1_c}}{\|W_{\cbar}\|_{\dot{H}^1_{\cbar}}} \geq 1+\delta' \quad \text{or}\quad \|u(t)\|_{\dot{H}^1_c} \geq (1+\delta') \|W_{\cbar}\|_{\dot{H}^1_{\cbar}}. \label{refined blowup energy-critical estimate}
\end{align}
\indent We also have for $\ep>0$ small enough,
\begin{align}
8\|u(t)\|^2_{\dot{H}^1_c} -\frac{4d\alpha^\star}{\alpha^\star+2} \|u(t)\|^{\alpha^\star+2}_{L^{\alpha^\star+2}} + \ep \|u(t)\|^2_{\dot{H}^1_c} \leq -c<0, \label{refined virial estimate energy-critical}
\end{align}
for any $t$ in the existence time. Indeed, 
\[
\text{LHS}(\ref{refined virial estimate energy-critical})= 4d\alpha^\star E_c(u(t)) + (8+\ep-2d\alpha^\star) \|u(t)\|^2_{\dot{H}^1_c}.
\] 
The conservations of mass and energy, $(\ref{refined blowup energy-critical estimate condition}), (\ref{refined blowup energy-critical estimate}), (\ref{property W_c 1})$ and $(\ref{property W_c 2})$ then yield
\begin{align*}
\text{LHS}(\ref{refined virial estimate energy-critical}) &\leq 4d\alpha^\star (1-\delta) E_{\cbar}(W_{\cbar}) + (8+\ep -2d\alpha^\star) (1+\delta')^2 \|W_{\cbar}\|^2_{\dot{H}^1_{\cbar}} \\
&= \frac{16}{d-2} (1-\delta) \|W_{\cbar}\|^2_{\dot{H}^1_{\cbar}} + \Big(-\frac{16}{d-2}+\ep\Big) (1+\delta')^2 \|W_{\cbar}\|^2_{\dot{H}^1_{\cbar}} \\
&= \|W_{\cbar}\|^2_{\dot{H}^1_{\cbar}} \Big[ \frac{16}{d-2}(1-\delta - (1+\delta')^2) + \ep (1+\delta')^2 \Big].
\end{align*}
By taking $\ep>0$ small enough, we prove $(\ref{refined virial estimate energy-critical})$. \newline
\indent Let us consider the case $xu_0 \in L^2$ satisfying $E_c(u_0)<E_{\cbar}(W_{\cbar})$ and $\|u_0\|_{\dot{H}^1_c}> \|W_{\cbar}\|_{\dot{H}^1_{\cbar}}$. By the standard virial identity $(\ref{global virial identity})$ and $(\ref{refined virial estimate energy-critical})$,
\[
\frac{d^2}{dt^2} \|x u(t)\|^2_{L^2} = 8 \|u(t)\|^2_{\dot{H}^1_c} -\frac{4d\alpha^\star}{\alpha^\star+2} \|u(t)\|^{\alpha^\star+2}_{L^{\alpha^\star+2}} \leq -c <0. 
\]
This shows that the solution blows up in finite time. \newline
\indent We now consider the case $u_0$ is radial, and satisfies $E_c(u_0)<E_{\cbar}(W_{\cbar})$ and $\|u_0\|_{\dot{H}^1_c}> \|W_{\cbar}\|_{\dot{H}^1_{\cbar}}$. Using the localized virial estimate $(\ref{localized virial identity})$, we have
\[
\frac{d^2}{dt^2}V_{\varphi_R} (t)\leq 8\|u(t)\|^2_{\dot{H}^1_c} - \frac{4d\alpha^\star}{\alpha^\star+2} \|u(t)\|^{\alpha^\star+2}_{L^{\alpha^\star+2}} + O \Big( R^{-2} + R^{-\frac{(d-1)\alpha^\star}{2}} \|u(t)\|^{\frac{\alpha^\star}{2}}_{\dot{H}^1_c} \Big).
\]
Using the fact $\frac{\alpha^\star}{2}=\frac{2}{d-2} \leq 2$, the uniform bound $(\ref{blowup energy-critical estimate})$ and $(\ref{refined virial estimate energy-critical})$, we see that for $R>1$ large enough,
\[
\frac{d^2}{dt^2}V_{\varphi_R} (t)\leq -c/2<0.
\]
Therefore, the solution must blow up in finite time. \newline
\indent The blowup of Theorem $\ref{theorem blowup NLS energy-critical inverse square radial}$ follows by the same argument as above and $(\ref{property W_c 1 radial})-(\ref{relation sharp sobolev embedding constant radial})$.
\section*{Acknowledgments}
The author would like to express his deep thanks to his wife-Uyen Cong for her encouragement and support. He would like to thank his supervisor Prof. Jean-Marc Bouclet for the kind guidance and constant encouragement. He also would like to thank the reviewer for his/her helpful comments and suggestions. 



\begin{thebibliography}{99}
\bibitem{Aubin} {\bf T. Aubin}, {\it Prol\`emes isop\'erim\'etriques et espaces de Sobolev}, J. Diff. Geom. 11 (1976), 573-598.

\bibitem{BoucletMizutani} {\bf J. M. Bouclet, H. Mizutani}, {\it Uniform resolvent and Strichartz estimates for Schr\"odinger equations with critical singularities}, To appear in Trans. Amer. Math. Soc. 2017.

\bibitem{BoulengerHimmelsbachLenzmann} {\bf T. Boulenger, D. Himmelsbach, E. Lenzmann}, {\it Blowup for fractional NLS}, J. Funct. Anal. 271 (2016), 2569-2603.

\bibitem{BurqPlanchonStalkerTahvildar-Zadeh} {\bf N. Burq, F. Planchon, J. Stalker, A. S. Tahvildar-Zadeh}, {\it Strichartz estimates for the wave and Schr\"odinger equations with the inverse-square potential}, J. Funct. Anal. 203 (2003), 519-549.

\bibitem{Cazenave} {\bf T. Cazenave}, {\it  Semilinear Schr\"odinger equations}, Courant Lecture Notes in Mathematics 10, Courant Institute of Mathematical Sciences, AMS, 2003.

\bibitem{ChoOzawa} {\bf Y. Cho, T. Ozawa}, {\it Sobolev inequalities with symmetry}, Commun. Contemp. Math. 11 (2009), No. 3, 355-365.

\bibitem{ChristWeinstein} {\bf M. Christ, I. Weinstein}, {\it Dispersion of small amplitude solutions of the generalized Korteweg-de Vries equation}, J. Funct. Anal. 100 (1991), No. 1, 87-109.

\bibitem{CsoboGenoud} {\bf E. Csobo, F. Genoud}, {\it Minimal mass blow-up solutions for the $L^2$ critical NLS with inverse-square potential}, Nonlinear Anal. 168 (2018), 110--129..

\bibitem{Dodson} {\bf B. Dodson}, {\it Global well-posedness and scattering for the focusing, energy-critical nonlinear Schr\"odinger problem in dimension $d=4$ for initial data below a ground state threshold}, preprint, \url{arXiv:1409.1950}, 2014.

\bibitem{DuyckaertsHolmerRoudenko} {\bf T. Duyckaerts, J. Holmer, S. Roudenko}, {\it Scattering for the non-radial 3D cubic nonlinear Schr\"odinger equation}, Math. Res. Lett. 15 (2008), No. 6, 1233-1250.

\bibitem{FangXieCazenave} {\bf D. Fang, J. Xie, T. Cazenave}, {\it Scattering for the focusing energy-subcritical nonlinear Schr\"odinger equation}, Sci. China Math. 54 (2011), No. 10, 2037-2062.

\bibitem{Glassey} {\bf R. T. Glassey}, {\it On the blowing up of solutions to the Cauchy problem for nonlinear Schr\"odinger equations}, J. Math. Phys. 18 (1977), 1794-1797.

\bibitem{HolmerRoudenko} {\bf J. Holmer, S. Roudenko}, {\it A sharp condition for scattering of the radial 3D cubic nonlinear Schr\"odinger equation}, Comm. Math. Phys. 282 (2008), No. 2, 435-467.

\bibitem{KalfSchminckeWalterWust} {\bf H. Kalf, U. W. Schmincke, J. Walter, R. Wust}, {\it On the spectral theory of Schr\"odinger and Dirac operators with strongly singular potentials}, in: Spectral Theory and Differential Equations, 182-226, Lect. Notes in Math. 448, Springer, Berlin, 1975.

\bibitem{KenigMerle} {\bf C. Kenig, F. Merle}, {\it Global well-posedness, scattering, and blowup for the energy-critical focusing nonlinear Schr\"odinger equation in the radial case}, Invent. Math. 166 (2006), 645-675.

\bibitem{KillipVisanZhang-unpublish} {\bf R. Killip, M. Visan, X. Zhang}, {\it The focusing energy-critical nonlinear Schr\"odinger equation with radial data}, Unpublished manuscript, 2007.

\bibitem{KillipVisan} {\bf R. Killip, M. Visan}, {\it The focusing energy-critical nonlinear Schr\"odinger equation in dimensions five and higher}, Amer. J. Math. 132 (2010), 361-424.


\bibitem{KillipMurphyVisanZheng} {\bf R. Killip, J. Murphy, M. Visan, J. Zheng}, {\it The focusing cubic NLS with inverse-square potential in three space dimensions}, Differential Integral Equations 30, No. 3-4 (2017), 161--206.

\bibitem{KillipMiaoVisanZhangZheng-sobolev} {\bf R. Killip, C. Miao, M. Visan, J. Zhang, J. Zheng}, {\it Sobolev spaces adapted to the Schr\"odinger operator with inverse-square potential}, Math. Z. 288 (2018), No. 3-4, 1273--1298.

\bibitem{KillipMiaoVisanZhangZheng-energy} {\bf R. Killip, C. Miao, M. Visan, J. Zhang, J. Zheng}, {\it The energy-critical NLS with inverse-square potential}, Discrete Contin. Dyn. Syst. 37 (2017), 3831-3866.

\bibitem{LiebLoss} {\bf E. H. Lieb, M. Loss}, {\it Analysis}, Graduate Studies in Mathematics 14, AMS, Providence, Rhode Island, 2001.

\bibitem{LuMiaoMurphy} {\bf J. Lu, C. Miao, J. Murphy}, {\it Scattering in $H^1$ for the intercritical NLS with an inverse-square potential}, J. Differential Equations 264 (2018), No. 5, 3174--3211.

\bibitem{Merle} {\bf F. Merle}, {\it Determination of blow-up solutions with minimal mass for nonlinear Schr\"odinger equations with critical power}, Duke Math. J. 69 (1993), No. 2, 427-454.

\bibitem{OgawaTsutsumi1} {\bf T. Ogawa, Y. Tsutsumi}, {\it Blow-up of $H^1$ solutions for the nonlinear Schr\"odinger equation}, J. Differential Equations 92 (1991), 317-330.

\bibitem{OgawaTsutsumi2} {\bf T. Ogawa, Y. Tsutsumi}, {\it Blow-up of $H^1$ solutions for the one dimensional nonlinear Schr\"odinger equation with critical power nonlinearity}, Proc. Amer. Math. Soc. 111 (1991), 487-496.

\bibitem{OkazawaSuzukiYokota-cauchy} {\bf N. Okazawa, T. Suzuki, T. Yokota}, {\it Cauchy problem for nonlinear Schr\"odinger equations with inverse-square potentials}, App. Anal. 91 (2012), 1605-1629.

\bibitem{OkazawaSuzukiYokota-energy} {\bf N. Okazawa, T. Suzuki, T. Yokota}, {\it Energy methods for abstract nonlinear Schr\"odinger equations}, Evol. Equ. Control Theory 1 (2012), 337-354.

\bibitem{Strauss} {\bf W. A. Strauss}, {\it Existence of solitary waves in higher dimensions}, Comm. Math. Phys. 55 (1977), No. 2, 149-162.

\bibitem{Suzuki-hartree} {\bf T. Suzuki}, {\it Energy methods for Hartree type equations with inverse-square potentials}, Evol. Equ. Control Theory 2 (2013), 531-542.

\bibitem{Suzuki-bounded} {\bf T. Suzuki}, {\it Critical case of nonlinear Schr\"odinger equations with inverse-square potentials on bounded domains}, Math. Bohem. 139 (2014), 231-238.

\bibitem{Talenti} {\bf G. Talenti}, {\it Best constant in Sobolev inequality}, Ann. Math. Pura. Appl. 110 (1976), 353-372.

\bibitem{Weinstein} {\bf M. Weinstein}, {\it Nonlinear Schr\"odinger equations and sharp interpolation estimates}, Comm. Math. Phys. 87 (1983), 567-576.

\bibitem{Weinstein86} {\bf M. Weinstein}, {\it On the structure and formation of singularities of solutions to nonlinear dispersive evolution equations}, Comm. Partial Differential Equations 11 (1986), 545-565. 

\bibitem{ZhangZheng} {\bf J. Zhang, J. Zheng}, {\it Scattering theory for nonlinear Schr\"odinger with inverse-square potential}, J. Funct. Anal. 267 (2014), 2907-2932.

\end{thebibliography}
\end{document}